\documentclass[hidelinks,onefignum,onetabnum]{siamart220329}

\usepackage{moreverb}
\usepackage{xspace}
\usepackage{array}
\usepackage{blkarray}
\usepackage{tabularx}
\usepackage{enumerate}
\usepackage{cleveref}
\usepackage{multirow}
\usepackage{paralist}
\usepackage{mathtools}
\usepackage{xcolor}
\usepackage{bm}
\usepackage[normalem]{ulem}
\usepackage{siunitx}
\usepackage{subfigure}
\usepackage{graphicx}
\usepackage{epstopdf}
\usepackage{amsfonts}

\usepackage{accents}
\renewcommand*{\dot}[1]{\accentset{\mbox{\large\bfseries .}}{#1}}
\renewcommand*{\ddot}[1]{\accentset{\scalebox{.8}{\mbox{\large\bfseries ..}}}{#1}}
\renewcommand*{\dddot}[1]{\accentset{\scalebox{.8}{\mbox{\large\bfseries ...}}}{\,#1\,}}
\renewcommand*{\widehat}[1]{\accentset{\scalebox{1}[.4]{$\wedge$}}{#1}}

\newtheorem{alg}{Algorithm}[section]

\newenvironment{algo}[1]
   {\renewcommand{\>}{\hspace*{1.5em}}%
   \begin{alg}[#1]\sf~\\%
   }
   {
\end{alg}}

\newcommand\nc[1]{\newcommand{#1}}
\nc\rnc[1]{\renewcommand{#1}}
\usepackage{listings}
\usepackage{xcolor}
\usepackage{stmaryrd}
\usepackage{eqparbox,calc}
\usepackage[nobreak=true]{mdframed}

\definecolor{myred}{rgb}{0.50, 0. 0.}
\definecolor{keywordscolor}{rgb}{0,0,0.50}

\definecolor{solarized@base03}{HTML}{002B36}
\definecolor{solarized@base02}{HTML}{073642}
\definecolor{solarized@base01}{HTML}{586e75}
\definecolor{solarized@base00}{HTML}{657b83}
\definecolor{solarized@base0}{HTML}{839496}
\definecolor{solarized@base1}{HTML}{93a1a1}
\definecolor{solarized@base2}{HTML}{EEE8D5}
\definecolor{solarized@base3}{HTML}{FDF6E3}
\definecolor{solarized@yellow}{HTML}{B58900}
\definecolor{solarized@orange}{HTML}{CB4B16}
\definecolor{solarized@red}{HTML}{DC322F}
\definecolor{solarized@magenta}{HTML}{D33682}
\definecolor{solarized@violet}{HTML}{6C71C4}
\definecolor{solarized@blue}{HTML}{268BD2}
\definecolor{solarized@cyan}{HTML}{2AA198}
\definecolor{solarized@green}{HTML}{859900}

\makeatletter
\def\rf#1{(\@rf#1,.)}
\def\@rf#1,{\ref{eq:#1}\@ifnextchar . {\@endrf}{, \@rf}}
\def\@endrf.{}
\makeatother

\nc\rfr[2]{(\ref{#1}--\ref{#2})}

\nc\ds{\displaystyle}
\nc\TW{\textwidth}
\nc\s[1]{\mbox{\footnotesize$#1$}}
\nc\npp{{n{+}1}} 

\nc\apref[1]{Appendix~\ref{ap:#1}\xspace}
\nc\dfref[1]{Definition~\ref{df:#1}\xspace}
\nc\exref[1]{Example~\ref{ex:#1}\xspace}
\nc\lmref[1]{Lemma~\ref{lm:#1}\xspace}
\nc\scref[1]{Section~\ref{sc:#1}\xspace}
\nc\ssref[1]{Subsection~\ref{ss:#1}\xspace}
\nc\thref[1]{Theorem~\ref{th:#1}\xspace}
\nc\fgref[1]{Figure~\ref{fg:#1}\xspace}
\nc\tbref[1]{Table~\ref{tb:#1}\xspace}
\nc\apitref[1]{$\Sigma$\ref{it:#1}}
\nc\itref[1]{(\ref{it:#1})\xspace}

\nc\page{p.\@\xspace}
\nc\pageS{pp.\@\xspace}

\nc\cf{cf.\@\xspace}
\nc\eg{e.g.\@\xspace}
\nc\etc{etc.\@\xspace}
\nc\wrt{w.r.t.\@\xspace}

\rnc\~[1]{\mathbf{#1}} 
\rnc\`[1]{\underline{#1}}
\nc\ol[1]{\overline{#1}}
\nc\asg{\mathop{\scalebox{.7}{$\pmb{\leftarrow}$}}}
\rnc\:{\!:\!} 
\nc\ninf{-\infty}
\rnc\d{\textrm{d}}
\nc\cs{\textrm{cs}}
\nc\Dbd[2]{\frac{\d#1}{\d#2}}
\nc\tDbd[2]{\d#1/\d#2}
\nc\dif{\mathfrak{d}}
\nc\dof{\mathop{\textrm{\sc dof}}}
\nc{\f}{F}
\nc\kmax{k_\textrm{max}}
\nc\bmx[1]{\begin{bmatrix}#1\end{bmatrix}}
\nc\mx[1]{\begin{matrix}#1\end{matrix}}
\nc\tightmx[1]{\begin{array}{r@{\,}c@{\,}l}#1\end{array}}
\nc\smallbmx[1]{\bigl[\begin{smallmatrix}#1\end{smallmatrix}\bigr]}
\nc\fp{\dot{\f}}
\nc\up{\dot{u}}
\nc\vp{\dot{v}}
\nc\fh{\widehat{\f}}
\nc\Jaco{\ol{\~J}}
\nc\Jach{\widehat{\~J}}
\nc\lam{\lambda}
\nc\xh{\widehat{x}}
\nc\yh{\widehat{y}}
\nc\sigmah{\widehat{\sigma}}
\nc\Sigmah{\widehat{\Sigma}}
\nc\sigmao{\ol{\sigma}}
\nc\Sigmao{\ol{\Sigma}}
\nc\xp{\dot{x}}
\nc\xhp{\dot{\xh}}
\nc\yp{\dot{y}}
\nc\xsij{\der{x_j}{\sij}}
\nc\xcd[2]{\der{x_{#2}}{d_{#2}-c_{#1}}}
\nc\xcidj{\xcd ij}
\rnc\_[1]{#1\makebox[0pt][r]{\rule[-.35ex]{.4em}{.3pt}}}
\nc\DDOT[1]{\mathop{\odot_{#1}}}
\nc\DDOTi{\DDOT{}}
\nc\Def{:=}
\nc\eqntxt[1]{\quad \text{#1} \quad}
\nc\F{\mathbb{F}} 
\nc\R{\mathbb{R}} 
\nc\x{\times}
\nc\dtimes{\!\cdot\!}
\nc\tsexpand[1]{#1_0 + #1_1 s + #1_2 s^2 + \cdots} 
\nc\tsdexpand[1]{1#1_1 \!+\! 2#1_2 s \!+\! 3#1_3 s^2 \!+\! \cdots} 
\nc\lni{r} 
\nc\tend{t_\text{end}}
\nc\Smethod{$\Sigma$-method\xspace}
\nc\calT{\mathcal{T}}
\nc\sij{\sigma_{ij}}
\nc\sii{\sigma_{ii}}
\nc\fo{\ol{\f}}
\nc\xo{\ol{x}}
\nc\T{{\scalebox{.6}{$\top$}}} 
\nc\+{^\T} 
\nc\Val{\mathop{\text{Val}}}
\nc\out{\textsl{out}}
\nc\nrm[1]{|#1|}
\nc\norm[1]{\|#1\|}
\nc\unitv{\~u}
\nc\range[3]{#1 {=} #2\:#3}
\nc\lrange[3]{#2 \le #1 \le #3}
\nc\rng[3]{#1 = #2\:#3}
\nc\Setbg[1]{ \bigl\{ #1 \bigr\}}
\nc\Set[1]{ \{ #1 \} }
\nc\NumSet[1]{ {\mathbb #1 }}

\nc\ddt[1]{\frac{\d#1}{\d t}}
\nc\tddt[1]{\tfrac{\d#1}{\d t}}
\nc\der[2]{ {#1}^{\!(#2)} }
\nc\grad{{\nabla}}
\nc\diff{\mathop{\rm diff}}
\nc\dbd[2]{\frac{\partial #1}{\partial #2}} 
\nc\tdbd[2]{\partial #1 / \partial #2}  

\nc\taylcv[2]{ #1_{#2} }
\nc\taylcg[2]{ #1_{#2}^* }
\nc\taylc [3]{ (#1_{#2})_{#3} }
\nc\tc    [2]{ #1_{#2}}  
\nc\tcf   [2]{ (#1)_{#2}}
\nc\tcff  [3]{ ({#1}_{#2})_{#3} }
\nc\tcfbg [2]{ \bigl( #1 \bigr)_{#2} }
\nc\tcflr [2]{ \left(#1\right)_{#2} }
\nc\tbeg{t_0}

\nc\addop{\text{\li{ADD}}\xspace}
\nc\subop{\text{\li{SUB}}\xspace}
\nc\mulop{\li{MUL}\xspace}
\nc\divop{\li{DIV}\xspace}

\nc\odets{{\tt odets}\xspace}
\nc\odestiff{\li{ode15s}\xspace}
\nc\odeen{\li{ode89}\xspace}
\nc\cl{code-list\xspace}
\nc\IVP{initial value problem\xspace}
\nc\matlab{{\sc Matlab}\xspace}
\nc\tcs{TCs\xspace}
\nc\TS{Taylor series\xspace}
\nc\li[1]{\lstinline[basicstyle=\ttfamily]{#1}}
\nc\baos{BAOs\xspace}
\nc\sode{sub-ODE\xspace}
\nc\sodes{sub-ODEs\xspace}
\nc\odefun{\li{odefun}\xspace}
\nc\RRs{recurrence relations\xspace}
\nc\Same{S-\-am\-en\-able\xspace}
\nc\sigmx{signature matrix\xspace}
\nc\sigmxs{signature matrices\xspace}
\nc\sysJ{system Jacobian\xspace}

\nc\kind{\li{kind}\xspace}
\nc\ode{\li{ODE}\xspace}
\nc\idp{\li{IDP}\xspace}
\nc\algk{\li{ALG}\xspace}
\nc\sub{\li{SUB}\xspace}

\nc\mode{\li{mode}\xspace}
\nc\rr{\text{\li{RR}}\xspace}
\nc\ri{\text{\li{RI}}\xspace}
\nc\ir{\text{\li{IR}}\xspace}
\nc\mi{\li{I}\xspace}
\nc\mr{\li{R}\xspace}
\nc\rp{\li{RP}\xspace}
\nc\pr{\li{PR}\xspace}

\nc\opcode{\li{operation}\xspace}
\nc\refo{\text{\li{ref1}}\xspace}
\nc\reft{\text{\li{ref2}}\xspace}
\nc\imm{\text{\li{immediate}}\xspace}

\nc\val{\li{value}\xspace}
\nc\oper{\li{subODEfcn}\xspace}
\nc\pref{\li{paramref}\xspace}

\nc\pline{\texttt{P-line}\xspace}
\nc\plist{\texttt{P-list}\xspace}

\nc\NC[1]{\xspace{\color{red}#1}\xspace}
\definecolor{darkerred}{rgb}{1,0,0}
\definecolor{darkgreen}{rgb}{0.0, 0.39, 0.0} 
\nc\NN[1]{{\color{darkerred}#1}}
\nc\NR[2]{\sout{{\color{red}#1}} {\color{blue}#2}}

\nc\term[1]{\textsf{\slshape#1}\xspace}

\newsiamthm{example}{Example}
\newsiamremark{remark}{Remark}

\headers{Sub-ODE{s} Simplify Taylor Series Algorithms for ODE{s}}{N. Nedialkov and J. Pryce}

\title{Sub-ODEs Simplify Taylor Series Algorithms for Ordinary Differential Equations\thanks{Submitted to the editors December 8, 2024.
\funding{The first author acknowledges the support of the Natural Sciences and Engineering Research Council of Canada (NSERC), FRN RGPIN-2019-07054.}}}

\author{Nedialko S.~Nedialkov\thanks{
Department of Computing and Software, 
McMaster University, Hamilton, Ontario, Canada,
\email{nedialk@mcmaster.ca}.}
\and John D.~Pryce\thanks{
School of Mathematics,
Cardiff University, Cardiff, Wales, UK,
\email{prycejd1@cardiff.ac.uk}}
}

\begin{document}
\maketitle

\begin{abstract}
  A Taylor method for solving an ordinary differential equation initial-value problem $\xp = f(t,x)$, $x(\tbeg) = x_0$, computes the Taylor series (TS) of the solution at the current point, truncated\linebreak to some order, and then advances to the next point by summing the TS with a suitable step size.
  
  A standard ODE method (e.g.~Runge--Kutta) treats function $f$ as a black box, but a Taylor solver requires $f$ to be preprocessed into a \cl of elementary operations that it interprets as operations on (truncated) TS.
  The trade-off for this extra work includes arbitrary order, typically enabling much larger step sizes.
  
  For a standard function, such as $\exp$, this means evaluating $v(t)=\exp(u(t))$, where $u(t),v(t)$ are TS.
  The \sode method applies the ODE $\d v/\d u=v$, obeyed by $v=\exp(u)$, to {\em in-line} this operation as $\vp=v\up$.
  This gives economy of implementation: each function that satisfies a simple ODE  goes into the ``Taylor library'' with a few lines of code---not needing a separate recurrence relation, which is the typical approach.
  
  Mathematically, however, the use of \sodes generally transforms the original ODE into a differential-algebraic system, making it nontrivial to ensure a sound system of recurrences for Taylor coefficients.
  We prove that, regardless of how many sub-ODEs are incorporated into $f$, this approach guarantees a sound system.

  We introduce our \sode-based \matlab ODE solver and show that its performance compares favorably with solvers from the \matlab ODE suite.
\end{abstract}




\section{Introduction}
We are concerned with computer solution of an initial-value problem (IVP) for an ordinary differential equation (ODE) system:
\vspace{-1ex}
\begin{align}\label{eq:mainode}
  \xp = f(t,x),\quad  x(\tbeg) = x_0,
\end{align}
where $x=x(t)$ is a vector function of $t$, dot $\dot{\rule{0pt}{1ex}}$ means $\d/\d t$, and $f$ is built from the four basic arithmetic operations (BAOs) and suitable standard functions.
We exclude here non-differentiable functions like absolute value, $\min$ and $\max$, and conditional structures such as ``if'' statements.

A Taylor method for \rf{mainode} computes the \TS (TS), to some order, of the solution through the current computed point, and sums this with a suitable step size to move to the next point.

These methods go back centuries. However multistep and Runge--Kutta methods, since their invention in the 1800s, have been generally more popular for numerical work, because they treat the function $f$ as simply a black box that converts numeric input to numeric output.
A Taylor method cannot do this---it must know how $f$ is built, i.e.\ know $f$ as an expression or as code.

Conceptually the Taylor process seems straightforward: change $f$'s inputs from numbers to polynomials, and use the formulae for polynomial addition, multiplication, etc., ending with a polynomial output.
However, by the nature of an ODE, one must build up one Taylor order $k$ at a time from $0$ to the chosen order $p$.
Hence the solution process involves first encoding $f$ into a {\em\cl}, representing a set of recurrences that are repeatedly evaluated at ``run time".
There is an outer loop to step the solution from $t=\tbeg$ to some $\tend$, and an inner loop for $\range k0p$.

Despite this overhead, Taylor methods have various advantages:
\begin{compactitem}[$\bullet$]
  \item They allow any Taylor order $p$.
At small $p$, step size is too small; at large $p$, work is too much.
In double precision, typically $\lrange p{15}{20}$ is found most efficient overall.

  \item They are especially good, using arbitrary-precision arithmetic, for hyper-accurate work such as long-time modelling of solar system motion.
  In \cite{Barrio05b}, $p=180$ was typical.
  Complexity is polynomial in the digits of precision obtained \cite{corless2008polynomial}, which a fixed-order method cannot achieve.

  \item For validated solution by interval arithmetic, they have been popular because truncation error can easily be enclosed by interval evaluation of Taylor coefficients \cite{nedialkov1999validated}.
  
\end{compactitem}
\vspace{.5ex}

Established systems that compute Taylor coefficients include 
TIDES \cite{abad2012algorithm},\linebreak
ATOMFT \cite{Chang1994a},  
ADOL-C \cite{ADOLC},  
TAYLOR \cite{jorba2005software}, 
INTLAB \cite{Ru99a},  and
FADBAD++ \cite{FADBAD++} .

\subsection{Our contribution}

We have developed the {\em \sode method} as a key ingredient for Taylor solution of ODEs.
We have implemented a \matlab solver that uses it.
This paper explains and motivates the method, proves in a precise sense that it always works, and illustrates it by numerical experiments.

\subsubsection{Sub-ODEs}
These are a way to handle standard functions, that applies to any function that itself obeys a simple ODE.
Most functions met in practice are in this class, e.g.\ $\exp, \log, \text{sqrt}$.
The Gamma function, for example, is not in this class.

Take $\exp$ for example.
When used in a Taylor solver, its meaning is to compute
\vspace{-1.5ex}
\begin{align}\label{eq:v_expu}
  v(t)=\exp(u(t))
\end{align}

\vspace{-1ex}

\noindent where $u(t)$ is a \TS (both input and output truncated to order $p$).
Most Taylor implementations write an individual recurrence for this, maybe as a separate function.
Instead, we use the fact that $v=\exp(u)$ obeys the ODE $\d v/\d u = v$, so when $u,v$ are both functions of $t$ we have $\vp = v\up$ by the chain rule.
This is inserted as {\em \sode lines} of the \cl and executed with the rest.

Does this save (or cost) work at run time?
No---the cost of \rf{v_expu} is pretty much fixed by its mathematical algorithm.
If separate recurrence, and insertion of $\vp = v\up$, are both well written, they might give identical floating point operation counts.
But a \sode is far simpler to implement.
For $\exp$ in our solver it just consists of writing this one-line definition in class \li{CLitem} (Code-List item) that generates the \cl:
\vspace{-1.5ex}
\begin{lstlisting}
  function v = exp(u)
    v = subODE("exp", @exp, @(u,v)v, u);
  end
\end{lstlisting}\vspace{-1ex}
This causes insertion of one \cl line at each occurrence of $\exp$ in $f$.
The \sode versions of the 40-odd other standard functions one would wish to include are not much more complicated, e.g. $v=\sqrt{u}$, for which we use the ODE $\d v/\d u = \frac12 v/u$, has such a one-line \li{CLitem} definition, which inserts three \cl lines at each occurrence.

By contrast with writing 40-odd separate recurrences, the \sode approach makes it simpler to write, maintain and extend a Taylor standard function library.

Executing the \cl requires the implementation of the following:
\begin{itemize}
  \item Addition, subtraction, multiplication, and division of Taylor polynomials.
  \item The operator (derived later) to handle a \sode, which we denote%
  \footnote{For a spoken word we suggest {\em subo} as in ``sub'' followed by ``oh''.}
  by $\DDOTi$.
\end{itemize}
These five building blocks form a {\em kernel} for TC computation.
The \cl of $f(t,x)$ is {\em rational} in the sense that the differential relations described are built only of BAOs.
The \term{base function} of any standard function---the computer library's numeric version---is used by the $\DDOTi$ operator at Taylor order $k=0$, to initialise a recurrence.
All other operations at $k=0$, and all for $k>0$, are BAOs.

This economy of implementation has other advantages.
Code performance optimisation can focus on the kernel.
And, to find sensitivities with respect to an initial condition or parameter, the chain rule of AD only needs applying to the kernel.

\subsubsection{Do \sodes work?\!\!}\label{ss:sodeswork}

No matter how many standard function calls are replaced by \sodes in the \cl, are the resulting recurrences \term{sound} at run time, in the sense that no value is ever asked for before it is computed?
One might be dubious of, say, $\log(\cos(u))$, where the output of one \sode is input to another.

As \exref{simple} shows, inserting a \sode almost always turns an ODE into a \term{differential-algebraic equation} (DAE) system.
We answer our question above in terms of \term{structurally amenable} (\Same) DAEs, whose theory is summarised in \apref{smethod}.

Specifically, we consider a DAE in dependent variables $x_j=x_j(t)$, $\range j1n$
\vspace{-1.5ex}
\begin{align}\label{eq:maindae}
  \f_i(t,\mbox{the $x_j$ and derivatives of them}) = 0, \quad\range i1n.
\end{align}
We refer to it as the DAE $\f=0$, of size $n$ where $\f =(\f _1,\ldots,\f_n)$.

\scref{DAEview} proves three key results, which \scref{applyODE} applies to the basic fact that an ODE $\xp = f(t,x)$, converted to a DAE in the obvious way,
$\f := \xp - f(t,x) = 0$,
is \Same.
The key results are that these operations on $\f$ preserve S-amenability:
\begin{enumerate}[(a)]
  \item\label{it:res1} Replace an expression $\psi$ by a new variable $u$, and add a new equation $0 = u - \psi$.
 \item \label{it:res2} Differentiate one of $\f$'s equations.
 \item \label{it:res3} Replace equation $0{=}\vp-g'(u)\up$, where $g$ obeys $g'(u) {=} h(u,g(u))$, by $0{=}\vp-h(u,v)\up$.
\end{enumerate}
\smallskip
Inserting one \sode can be broken into at most two uses of \itref{res1} followed by one of \itref{res2} and one of \itref{res3}.
So starting with the ODE and inserting any number of \sodes, we end with an \Same DAE (\thref{main1}).
Every \Same DAE has a \term{standard solution scheme}---an always sound recurrence for expanding a solution in TS.
This turns out identical with what the \cl does, proving \sodes work.

\subsection{Back story}

Automatic solution of IVPs for ODEs by TS expansion goes back (at least) to Moore (1966) \cite{moore}.
So does the idea of converting an ODE to an equivalent rational ODE---in \cite[\S11.2]{moore} is an example where $\cos$, $\exp$, and $\log$ occur, and are eliminated.
In the DAE context, Corless and Ilie (2008)~\cite{corless2008polynomial} note explicitly that one can convert a DAE to an equivalent one using only BAOs.
But to our knowledge, Neidinger (2013)~\cite{Neidinger2013} is the first person to give an embedding {\em method} that can in principle be automated---an operator he calls \li{ddot}, nearly equivalent to our $\DDOTi$.

Both Moore and Neidinger only consider \RRs (RRs) for calculating 
Taylor coefficients (\tcs), and not a differential system that these might represent; while Corless--Ilie treat the differential system, but not (at that point of their exposition) the resulting RRs, nor with an automated way to do the conversion.
None notes that barring special cases, the result of thus converting an ODE is always a DAE.
To recast it as an ODE, one must differentiate one or more equations, and algebraically manipulate the result---Moore does this in his example without noting it is unavoidable.

Nor do they note this adds extra degrees of freedom (DOF) that need initial values.
In an IVP code's stepping process, these must be found and applied at each step.
By contrast it is inherent in the $\DDOTi$ operator that what DOF it adds by differentiation, it removes by its built-in IVs, leaving their number unchanged.

\subsection{Article Outline}

\scref{method} introduces \TS solution, \cl{}s and \sodes, with examples. 
It defines the \sode operator and outlines the computation of TCs when \sodes are present.
\scref{DAEview} gives precise statements of \itref{res1}--\itref{res3} in \ssref{sodeswork}, and proofs.
\scref{applyODE} applies these to the insertion of sub-ODEs in an ODE.
\scref{results} presents numerical results from our \matlab solver, comparing it for accuracy and performance with the \matlab ODE suite's solvers. 
\scref{concl} outlines some design and implementation issues of turning theory into software, and plans for future work.
\apref{smethod} summarises \Same DAE theory.

\section{The Method}\label{sc:method}

\subsection{Taylor series solution}

For an ODE $\xp=f(t,x)$ with a suitably smooth $f$, how to find the TS of a solution $x(t)$ is obvious in principle.
Given an initial value $x(\tbeg)=x_0$, we have $\xp(\tbeg) = f(\tbeg,x_0)$.
For $\ddot{x}(\tbeg)$, differentiate $f$ using the chain rule:
\[ \ddot{x}(\tbeg) = \tdbd{f}{t}(\tbeg,x_0) +  \tdbd{f}{x}(\tbeg,x_0) \xp(\tbeg). \]
This gives the first three TS terms, $x(\tbeg+s) = x(\tbeg) + \xp(\tbeg)s/1! + \ddot{x}(\tbeg)s^2 /2!+ \cdots$; to find more, keep differentiating.
This proves TS solution is possible but does not give a practical scheme.
For that, use {\em automatic} (or algorithmic) {\em differentiation}, AD, to derive an ODE-specific recurrence from how $f$ is built of elementary operations.
E.g.,\ to solve $\xp = x^2$, $x(\tbeg)=x_0$, assume a series $x(\tbeg+s) = \tsexpand{x}$.~Then
\begin{align*}
  x^2 &= x_0^2 + (2x_0x_1)s + (2x_0x_2+x_1^2)s^2 + \cdots, \\
  \xp &= x_1 + 2x_2s + 3x_3s^2 + \cdots.
\end{align*}
Starting with a known $x_0$, and equating these term by term, gives $x_1=x_0^2$, then $x_2=(2x_0x_1)/2$, $x_3=(2x_0x_2+x_1^2)/3$, and so on.

In general, denote the $k$th TC of a function $u$ of a real variable $t$ at $\tbeg$ by
\begin{align}\label{eq:TCdef}
  \tc uk = \der{u}{k}(\tbeg)/k!\,.
\end{align}
Our method is built on the fundamental recurrences for the basic arithmetic operations: given TCs of $u$ and $v$ to order $k$, the $k$th TCs for $u\pm v$, $u\cdot v$ and $u/v$  are
\vspace{-2.5ex}
\begin{align}\label{eq:arithops}
\begin{split}
  \tcf{u\pm v }{k} &= \tc{u}{k} \pm  \tc{v}{k}, \qquad
  \tcf{u\cdot v}{k}  = \sum_{r = 0}^{k} \tc{u}{r} \tc{v}{k-r},\\[-1ex]
  \tcf{u/v}{k}    &= \frac{1}{\tc{v}{0}} \bigl(\tc{u}{k} - \sum_{r = 0}^{p-1} \tc{v}{k-r} \tcf{u/v}{r} \bigr),\quad (\tc{v}{0} \ne 0).
\end{split}
\end{align}

\subsection{Notation, \cl{}s}\label{ss:codelists}

The scalar field $\F$ will be the real numbers $\R$ but it could equally be the complex numbers.
An $m$-vector means an element of $\F^m$ for some $m\ge1$.
We do not distinguish row from column vectors, except when needed for linear algebra.
The notation $r\:s$ for integer $r,s$ means either the {\em list} (sequ\-ence) of integers $r,r+1,\ldots,s$ or the {\em set} of them according to context.
Let $\tbeg\in\R$ be a given base point for TS expansion.

Consider an ODE \rf{mainode} of size $n$, i.e.\ $x$ is an $n$-vector, whose elements $x_1,\ldots,x_n$ are the  {\em state variables}.
By our convention $x_\npp$ is an alias for $t$.
A {\em\cl} for the ODE is a sequence of assignments that, starting from input-values of $x_1,\ldots,x_n,t$, set {\em intermediate variables} $x_{n+2},\ldots, x_N$.

Specifically, a relation $\prec$ on $1\:N$ is given.
If $i \prec j$, we say ``$i$ is used by $j$'' or, interchangeably, ``$x_i$ is used by $x_j$''.
Its inverse $j\succ i$ is ``$j$ uses $i$'' or ``$x_j$ uses $x_i$''.
Functions $\phi_{n+2},\ldots,\phi_N$ are given such that the \cl assignments are
\begin{align}\label{eq:cl_assign1}
  x_j &= \phi_j(x_{\prec j}) \quad\text{for $j=(n+2)\:N$},
\end{align}
where $x_{\prec j}$ is the vector (in some standard order) of those $x_i$ such that $i\prec j$.
The assignments are done in this order, so the \cl has to be \term{sound}, meaning
\begin{align*}
  \text{$i\prec j$ implies $i<j$},
\end{align*}
i.e.\ each assignment only uses already set values.
State variables and $t$ have values at the start, so $x_{\prec j}$ is empty for $j{=}1\:\npp$.
Other $x_{\prec j}$ might be empty: then $\phi_j$ is a constant.
\rf{cl_assign1} implies {\em static single assignment} form---each $x_j$ gets a value just once.

Finally the \cl specifies a list of indices $\out(1),\ldots,\out(n)$ in $1\:N$ such that $\xh_i=x_{\out(i)}$ form the {\em output vector} $\xh$.
By definition $\xp$ equals $\xh$, thus defining the ODE.
This matches a typical ODE solver interface, for example in the \matlab ODE suite one implements \rf{mainode} as a function \li{dydt = odefun(t,y)} where \li{odefun} is the $f$ and \li{dydt} is the $\xh$ that equals $\xp$.

The ``level of resolution'' of the $\phi_j$ in a \cl is flexible.
In detail, our implementation breaks $f$ into just the BAOs and the $\DDOTi$ operator.
At the other extreme, ODE \rf{mainode}, broken into into $\xh=f(t,x)$ followed by $\xp=\xh$, is a perfectly good \cl.

\subsection{Sub-ODEs}\label{ss:subODEs}

A \sode relates two objects that we call $u$ and $v$.
We switch between two uses:
\begin{itemize}[--]
  \item $u$ is a scalar, $v$ is an $m$-vector;
  \item $u$ is a scalar function $u(t)$, $v$ is an $m$-vector function $v(t)$, of real variable $t$.
\end{itemize}
Often, when considering spaces of functions and operators on them, one omits the $(t)$.
Where confusion could occur, we underline them $\_u,\_v$ when denoting functions of $t$.

\subsubsection{The \sode concept}

Suppose $v = g(u)$ is an $m$-vector function of scalar $u$, and we can express $\frac{\d v}{\d u}$ as $ h(u,v)$, i.e.\ $h$ defines an ODE that $g$ solves:
\begin{align}\label{eq:huv}
    g'(u) = h(u,g(u)).
\end{align}
For instance if $v = g(u) = u^c$, where $c$ is a constant, we can take $h(u, v) = cv/u$, since 
\[
  h(u, g(u)) = c u^c/u = c  u^{c-1} = g'(u).
\]
Suppose $u,v$, above, are variables in the \cl of ODE \rf{mainode}: then $u,v$ denote $\_u(t),\_v(t)$, and $v=g(u)$ means $\_v(t)=g(\_u(t))$.
  Using $\Dbd{v}{t} = \Dbd{v}{u} \Dbd{u}{t}$, we see \rf{huv} implies
\begin{equation}\label{eq:huvd}
  \vp(t) = h\bigl(\_u(t),\_v(t)\bigr) \up(t), \quad\text{or}\quad \vp = h(\_u,\_v) \up \quad\text{for short, termed a \sode}.
\end{equation}
The method replaces each assignment $v=g(u)$ by the code of its $\vp = h(\_u,\_v) \up$.
\begin{example}\rm\label{ex:subodes}
Examples of \sodes produced by standard functions are:
\begin{align*}
\begin{tabular}{rl@{\qquad has }l@{\qquad $h(u,v)=$ }l}
  (i) &$v = \exp(u)$                      &$\vp = \_v \up$,                &$v$;
   \\[1ex]
 (ii) &$v = u^c$, ($c$ constant)          &$\vp = (c\_v/\_u) \up$,        &$cv/u$; \\
 \multicolumn4l{\text{and, defining $\cs(u) = \smallbmx{\cos(u)\\ \sin(u)}$,}} \\[1ex]
(iii) &$v = \bmx{v_1\\v_2} = \cs(u)$      &$\vp = \bmx{-\_v_2\\ \_v_1} \up$, &$\bmx{-v_2\\ v_1}$.
\end{tabular}
\end{align*}
\end{example}

\subsubsection{Sub-ODEs make a DAE}

It is not obvious that this method makes sense, because usually the resulting \cl represents not an ODE, but a DAE.

\begin{example}\rm\label{ex:simple}
  We illustrate this by a simple \sode scenario, and a concrete example.
  Take a scalar ODE $\xp=f(t,x)$.
  Write it, see \ssref{codelists}, as $\xh=f(t,x)$ followed by $\xp=\xh$.
  Let $f$ use a standard function $g$, such that the ODE splits into
  \begin{align}\label{eq:huvscenario}
  u   &= f_1(t,x), &
  v   &= g(u), &
  \xh &= f_2(t,v), &
  \xp&= \xh. &&
  \shortintertext{Applying the \sode means replacing the second equation by}
  && \vp &= h(u,v) \up. \notag
\end{align}
Let the concrete example be the simple ODE
$\xp = e^{-x}$, with general solution $x = \log(t-C)$, $C$ an arbitrary constant.
(Assuming we only consider real solutions, $C$ is real and $x$ is defined on $C<t<\infty$.)
Following \rf{huvscenario}, split this into
\vspace{-1.5ex}
\begin{align*}
  u &= f_1(t,x) = -x, &
  v &= g(u) = \exp(u), &
  \xh &= f_2(t,v) = v, &
  \xp&= \xh. &&
\end{align*}

We now do as follows.
As it shows the general pattern more clearly, put $\xp=\xh$ at the start; insert $\exp$'s \sode which is $\vp = v \up$; since $\xh=v$, eliminate $v$. We get
\vspace{-1.5ex}
\begin{align}\label{eq:huvconcr2}
   \xp &= \xh, &
     u &= -x, &
  \xhp &= \xh \up \, . & &
\end{align}
Casting \rf{huvconcr2} as an ODE cannot use just algebraic manipulation---it needs $\up = -\xp$, i.e.\ one {\em must differentiate} the second equation.
Thus we have a DAE, not an ODE.
\end{example}

A fortiori, this must apply to scenario \rf{huvscenario}, hence to the general case of any number of \sodes.
\begin{remark}
If the ODE is $\xp = e^x$ instead of $\xp = e^{-x}$, we don't need intermediate variable $u$, so needn't differentiate anything to reduce to an ODE---this is an implicit ODE.
\scref{applyODE} shows just when the result of using \sodes is inherently a DAE.
\end{remark}

\subsubsection{\ldots but that's OK}\label{ss:subode_ok}

Nevertheless a recursion for the TCs of \rf{huvconcr2} at an initial point $(\tbeg,x_0)$ falls out at once.
Write $x(t) = x(\tbeg+s) = \tsexpand{x}$; similarly for $u,v,\xh$.
Inserting in \rf{huvconcr2} gives
\begin{align}\label{eq:huvconcr3}
\left\{\begin{aligned}
  \tsdexpand{x} &= (\tsexpand{\xh}), \\
    \tsexpand{u} &= -(\tsexpand{x}), \\
  \tsdexpand{\xh} &= (\tsexpand{\xh})(\tsdexpand{u}).
\end{aligned}\right.
\end{align}
To start the recursion, we set $x_0=$ {\em the user's IV}, and when $u_0$ is available, $v_0=g(u_0) = e^{u_0}$, the \term{built-in IV}\/ of the \sode.

Expanding and equating terms either side of \rf{huvconcr3} gives the following.
\begin{align}\label{eq:huvconcr4}
  \begin{aligned}
      k&=0 \\\hline\\[-3ex]
    x_0&=\text{user's IV} \\
    u_0&=-x_0 \\
  \xh_0&=\exp(u_0)=\text{built-in IV}
  \end{aligned}
  &&
  \begin{aligned}
      k&=1 \\\hline\\[-3ex]
    x_1&=\tfrac11 \xh_0 \\
    u_1&=-x_1 \\
    \xh_1&=\tfrac11 (\xh_0 u_1)
  \end{aligned}
  &&
  \begin{aligned}
      k&=2 \\\hline\\[-3ex]
    x_2&=\tfrac12 \xh_1 \\
    u_2&=-x_2 \\
    \xh_2&=\tfrac12 (\xh_1 u_1 + \xh_0 2u_2)
  \end{aligned}
  &&
  \begin{aligned}
    \cdots\\\hline\\[-3ex]
    ~ \\
    ~ \\
    ~
  \end{aligned}
\end{align}
where we do the $k=0$ assignments in order, then the $k=1$, and so on.

What has happened seems remarkable.
Just inserting the \cl of $h$ (which here just copies $\xh$, with no arithmetic involved) in the natural place, creates a sound recurrence, meaning one never calls for data before it has been computed.

The third equation of \rf{huvconcr3}, which implements \rf{huvd}, causes this minor miracle.
It ensures that for each $k\ge1$, value $\xh_k$ depends only on $u_i$ for $i\le k$, which are all available, and $\xh_i$ for $i$ strictly $<k$, also available.
Also when $h(u,v)=v$ is replaced by a general $h(u,v)$, with expansion $h_0 + h_1 s + h_2 s^2 + \cdots$, this remains true because $h_k$ depends only on $u_i,v_i$ for $i\le k$.

\begin{remark}
We believe Neidinger \cite[p.~590]{Neidinger2013} first noted the above property of \sodes explicitly.
His operation \verb`ddot` is essentially our $\DDOTi$, and he says (slightly edited) ``Notice that \verb`ddot` uses coefficients only from $1$ to $k-1$, which is important in order to avoid a self-referential definition~\ldots''.
\end{remark}

\begin{remark}
The purposes of \sode and main ODE are almost opposite. In the \sode, the base-function $g$ is {\em known}; we aim to ``lift'' it to the Taylor version, as a means to get the TS of the main ODE at the current point. In the main ODE, the numeric version is {\em unknown} in practice, and the Taylor version is used (by a TS-based solver) just as a means to get the numeric version.
\end{remark}

In the \cl for \rf{huvconcr2}, below, $\DDOTi$ is the \sode operator of \dfref{sodedef}. \\
\begin{tabular}{c|@{}c}
  Math &~As produced by our system \\\hline
  \rnc\arraystretch{0.9}
  $\begin{array}{rl}
  \\
    \xp &= \xh \\
      u &= -x \\
    \xh &= \xh\DDOT{\exp}u
  \end{array}$
  & \small\begin{tabular}{rcccrrl|c @{\hskip 0.1cm} c @{\hskip 0.1cm} l}
Line & Kind & Op & Mode & R1 & R2  & Imm & \multicolumn{3}{c}{Expression} \\ \hline 
\texttt{1} & \texttt{ODE} & \texttt{} & \texttt{R} & \texttt{3} & \texttt{} & \texttt{}  & $\dot x_{1}$ &$=$& $x_{3}$ \\ 
\texttt{2} & \texttt{ALG} & \texttt{sub} & \texttt{IR} & \texttt{} & \texttt{1} & \texttt{ 0}  & $x_{2}$ & $=$ & $ 0 - x_{1}$ \\ 
\texttt{3} & \texttt{SUB} & \texttt{exp} & \texttt{RR} & \texttt{3} & \texttt{2} & \texttt{}  & $ x_{3}$ &$=$ & $x_{3} \odot_\text{exp} x_{2}$ \\ 
\end{tabular}
\end{tabular}

A line in this table (excluding the last column) is a description of a \cl instruction. 
The  kind field gives the type of operation being performed; the operation field Op contains the name of a specific operation.
The mode says how operands are obtained: \texttt{I} means immediate (stored in Imm), and \texttt{R} means reference (in R1 or R2).
\begin{enumerate}[{Line} 1]
  \item is of  \texttt{ODE} kind, indicating $\out(1)=3$, i.e.\ the derivative is found on line 3.
  \item  is of \texttt{ALG} (algebraic) kind  with operation \texttt{sub} (subtraction), 
  which is used to implement the unary minus $-x_{1}$ as $0-x_{1}$.

  \item   is of \texttt{SUB} (\sode) kind;   \texttt{exp} is the function to be executed to obtain an initial value.  \texttt{R2} references  line 2, where the input $x_{2} = u$ is computed. \texttt{R1} refers to the line containing the output of line 3  (they are the same here).
\end{enumerate}
\smallskip

\subsection{The \sode operator}

We formally define the operator $\DDOTi$, which specifies both equation \rf{huvd} of a \sode, and its built-in IV.
Assume $h(u,v)$ is defined by an expression that uses only BAOs and makes $h$ an $m$-vector function of scalar $u$ and $m$-vector $v$.
Excluding points where division by zero occurs, it defines a smooth function $h: D\subseteq\R\x\R^m \to \R^m$, where $D$ is open (and dense). Hence the ODE
\vspace{-1.5ex}
\begin{align}\label{eq:hode1}
  \frac{\d u}{\d v} &= h(u,v)
\end{align}
\vspace{-3ex}\par\noindent
has a unique continuous $m$-vector solution $v=\gamma(u)$ through each point $(u_0,v_0)\in D$.
Let real $\tbeg$ be given, and $u = u(t)$ be a suitably smooth scalar function of real $t$ on a neighbourhood of $\tbeg$.
Operator $\DDOTi$ depends on $\tbeg$ as well as $h$, $\_u$ and $\gamma$, but our notation leaves this implicit.
\begin{definition}[Meaning of \sode operator]\label{df:sodedef}~
\begin{itemize}[--]
\item[\rm Indefinite $\DDOTi$.]
  Equation $\_v = h\DDOTi \_u$ means $\_v$ is an $m$-vector function of $t$ that, near $\tbeg$, obeys \rf{huvd} seen as an ODE for $\_v$ with given $\_u,h$.
  Since $\Dbd{v}{t} = \Dbd{v}{u} \Dbd{u}{t}$ we see that if
  \[
   \text{$\_v(t) = \gamma\bigl(\_u(t)\bigr)$ holds near $t=\tbeg$ for some solution $\gamma$ of \rf{hode1}},
  \]
  then $\_v = h\DDOTi\_u$. We assume $\gamma$ is such that $\gamma(\_u(\tbeg))$ is defined.
\item[\rm Definite $\DDOTi$.]
  The equation $\_v = h\DDOT{\gamma}\_u$ is $\DDOTi$ with a built-in IV.
  For the definition, $\gamma$ can be an arbitrary $m$-vector function of scalar $u$, and $\DDOT{\gamma}$ means that $\_v = h\DDOTi\_u$ and also $\_v(\tbeg) = \gamma(\_u(\tbeg))$.
  However it only gives consistent results as $\tbeg$ varies, if $\gamma$ is one of the solutions of \rf{hode1}, and we shall only use this case.
\end{itemize}
\end{definition}

From this comes the ``recurrence'' meaning of $\DDOTi$ (shown in line 3 of \rf{huvconcr3}).
Let $\_u$, $\_v$ and $\_h {=} h(\_u,\_v)$ have TS at $\tbeg$, so $\_u{=}\tsexpand{u}$, etc.
Then by \rf{huvd}
\vspace{-1ex}
\[ \tsdexpand{v} = (\tsdexpand{u})(\tsexpand{h}) \]
so {\em indefinite} $\_v = \_h\DDOTi \_u$ means $v_1 =\tfrac11 (1u_1 h_0)$; then $v_2 =\tfrac12 (1u_1h_1 + 2u_2 h_0)$; in general
\vspace{-1.5ex}
\begin{align}
  v_k &= \frac1k \sum_{i=1}^{k}iu_{i}h_{k-i} 
  \quad\text{for $k=1,2,\ldots$} \label{eq:ddot1}
\shortintertext{while {\em definite} $\_v = \_h\DDOT{\gamma} \_u$ means the same, plus the equation at $k=0$:}
  v_0 &= \gamma(u_0). \notag
\end{align}

\begin{example}\rm
  Consider \exref{subodes}. Let $A,B$ denote arbitrary constants.

\noindent(i) 
  Here $v=g(u)=e^u$.
  The ODE \rf{hode1} is $\tDbd{v}{u} = v$ with general solution $v = \gamma(u) = Ae^u$.
  The \sode \rf{huvd} is $\vp(t) = h\bigl(\_u(t),\_v(t)\bigr) \up(t)$ for given $\_u(t)$, with general solution $\_v(t) = A\exp(\_u(t))$, i.e.\ $\_v(t) = \gamma(\_u(t))$.
  So $h\DDOTi \_u$ is this family of solutions.
  
  The definite form $\_v = h\DDOT{\gamma} \_u$ adds the built-in IV $\_v(\tbeg) = \gamma(\_u(\tbeg))$. If $v=\gamma(u)$ is some arbitrary function of $u$, the resulting function $v(t)$ will be different for different $\tbeg$; but if $\gamma$ is a member of the general solution, say $v=A^*e^u$, this forces $A$ to be $A^*$, independent of $\tbeg$.
  In particular if $\gamma$ is $\exp$ (so $A=1$) then $\DDOTi$ forms $\exp(\_u(t))$.
  
  This is not going in circles. The $v=\exp(u)$ is the {\em base function} in the computer library, and the aim of \sodes is to ``lift'' this efficiently to the Taylor version $\_v(t)=\exp(\_u(t))$ which forms the TS of $\_v$ from that of $\_u$.
  
\noindent(ii) 
  Here the relevant ODE is $\tDbd{v}{u} = cv/u$ with general solution $v = Au^c$.
  Choosing $\gamma$ to be the standard power function $u^c$ we force $A=1$, and Taylor version $\_v(t)=\_u(t)^c$.
  
\noindent(iii) 
  Here the general solution of the $h$ ODE is $v = \bmx{A\cos u + B\sin u\\-A\sin u + B\cos u}$.
  Choosing base-function $\gamma(u)=\bmx{\cos(u)\\ \sin(u)}$ forces $A=1,B=0$, and Taylor version $\bmx{\cos(\_u(t))\\ \sin(\_u(t))}$.
\end{example}

In Case (iii) we might be interested in a different solution, say with $A=0,B=1$: the same \sode will do this, with a different built-in IV $\gamma$.
This is useful for \eg, Bessel functions, where the functions of the first, second and third kinds---$J_{\nu}$, $Y_{\nu}$ and $H_{\nu}$---satisfy the same ODE.
So their Taylor versions can be implemented by the same \sode, just with different base-functions.

\section{Sub-ODEs from a DAE viewpoint}\label{sc:DAEview}

\subsection{Discussion}

Using a \sode means replacing some standard function, occurring in the definition of an ODE or DAE, by an ODE that it satisfies.
\ssref{subODEs} showed that this converts an ODE to a DAE.
Why does this not cause problems for the Taylor coefficient recurrences?
The theory of \term{standard solution scheme} of a \term{structurally amenable} (\Same) DAE answers this question.
It is summarised in \apref{smethod}; a term defined there is shown \term{slanted} on first appearance.

Consider a DAE in $n$ dependent variables $x_j=x_j(t)$, of the form
\vspace{-1.5ex}
\begin{align}\label{eq:daedef1}
  \f_i(t,\mbox{the $x_j$ and derivatives of them}) = 0, \quad i=1\:n.
\end{align}
We assume functions are suitably smooth, so needed derivatives always exist.

As informal notation, denote by $\`x$ the vector of $t$, and all derivatives $\der{x_j}{k}$, $k{\ge}0$, $\range j1n$ that occur in the $\f_i$, in some standard order.
Then \rf{daedef1} can be written 
\vspace{-1.5ex}
\begin{align}\label{eq:daedef2}
  \f(\`x) = 0.
\end{align}

As we do it, using a \sode involves three stages.
\begin{compactenum}[(1)]
  \item Extract each relevant occurrence of function $g$ to an explicit assignment $v=g(u)$, making $u(t),v(t)$  state-variables ($u$ scalar, $v$ possibly vector) if they aren't already.
  \item Differentiate $v=g(u)$ to become $\vp=g'(u)\up$. 
  \item Replace $\vp=g'(u)\up$ by $\vp=h(u,v)\up$ for a suitable $h$.
\end{compactenum}
In case we started with an ODE, an initial stage (0) recasts it as a DAE, which is easily shown to be always \Same.
Then Subsections \ref{ss:extract}, \ref{ss:diffdae}, \ref{ss:introsode} show each of stages (1)--(3) preserves S-amenability.

\subsection{Extracting a subexpression preserves S-amenability}\label{ss:extract}

The first result is purely algebraic and deceptively simple: if a DAE \rf{daedef2} is \Same, it remains so when we extract a scalar expression $\psi(\`x)$ present in one or more components $\f_r(\`x)$ and replace it by a new variable $u$.

Replacing is optional: if there are several occurrences of the expression, one can replace all, some or none.
For example with $u=\psi(\`x) = x_1+\xp_2$, one can convert $(x_1+\xp_2)\log(x_1+\xp_2)$ to $u\log(u)$ or $u\log(x_1+\xp_2)$ or $(x_1+\xp_2)\log(u)$, or leave it alone.
The $\psi$ expression may contain derivatives of the $x_j$, but derivatives of $u$ are {\em forbidden}.
E.g.\ if we set $u=\xp_1$, we may not write $\ddot{x}_1$ as $\dot{u}$.

All such options are captured by saying there are functions $\fo_r(\`x,u)$ indexed over a set $R\subseteq\{1,\ldots,n\}$, such that for each $r\in R$, 
\begin{align}\label{eq:xtr0}
   \f_r(\`x) = \fo_r(\`x,\psi(\`x)).
\end{align}
This converts $\f(\`x)=0$ to a new DAE $\fo(\`x,u)=0$ of size $\npp$, which for each $r\in R$ replaces $0=\f_r(\`x)$ by
\vspace{-1.5ex}
\begin{align*}
  0 &= \fo_r(\`x,u) &&\text{which stays as equation $r$}, 
  \shortintertext{the $\f_i(\`x)$ for $i\notin R$ staying unchanged but now regarded as $\f_i(\`x,u)$, and adds}
  0 &= \phi(\`x,u) := u - \psi(\`x) &&\text{which becomes equation $\npp$}. 
\end{align*}

It is helpful to put an intermediate stage between the size $n$ DAE $\f(\`x)=0$ and the size $\npp$ DAE $\fo(\`x,u)=0$.
Namely, let $\fh=0$ of size $\npp$ have  its first $n$ functions the same as those of $\f=0$, but regarded now as functions of $(\`x,u)$, and $(\npp)$th function $\phi(\`x,u)$.

The \term{\sigmx} of $\fh$ is the $(\npp)\x(\npp)$ matrix
\vspace{-1.5ex}
\begin{align}\label{eq:xtr3}
  \Sigmah &=
  \left[\begin{array}{c|c}
  \Sigma & - \\\hline \sigma_{\psi} & 0
  \end{array}\right]
  \quad\text{where}\ \
    \parbox{.6\TW}
    {$\Sigma$ is \sigmx of $\f$; column $-$ is all $\ninf$; \\
    $\sigma_{\psi,j}$ is highest derivative order of $x_j$ in $\psi$.}
\end{align}
The DAEs $\fh=0$ and $\fo=0$ differ only in their $r$-components for $r\in R$:
\vspace{-1.5ex}
\begin{align}\label{eq:xtr4}
 (a)\ \  \fh_r(\`x,u)\Def\fo_r(\`x,\psi(\`x)) \quad\text{versus}\quad (b)\ \  \fo_r(\`x,u).
\end{align}

From \rf{xtr3}, the \term{HVTs} of $\Sigmah$ are precisely those of $\Sigma$ with the entry $(\npp,\npp)$ added.
It is easy to see the \term{\sysJ{}s} $\~J$ of $\f$ and $\Jach$ of $\fh$, derived from any \term{valid offsets}, are related by $\Jach = \smallbmx{\~J &0 \\ \x\x &1}$ for some entries $\x\x$ in the bottom row; so $\det\Jach=\det\~J\ne0$, and $\fh=0$ is \Same iff $\f=0$ is so.

The easy part of the argument is:
\begin{lemma}
  The DAEs $\f = 0$, $\fh = 0$ and $\fo = 0$ have the same solution set in the sense that if $x(t)$ solves $\f(\`x) = 0$ then $(x(t),u(t))$, where $u(t)=\psi(\`x(t))$, solves both $\fh(\`x,u) = 0$ and $\fo(\`x,u) = 0$.
  Conversely, if $(x(t),u(t))$ solves either of $\fh(\`x,u) = 0$ or $\fo(\`x,u) = 0$ then it solves the other, and $x(t)$ solves $\f(\`x) = 0$.
\end{lemma}
\begin{proof}
  Let $x(t)$ solve $\f=0$. As $\fh=0$ differs only in the extra equation that says $u=\psi(\`x)$, appending to $x(t)$ the component $u(t)=\psi(\`x(t))$ gives a solution of $\fh=0$.
  By definition the latter says $u(t)=\psi(\`x(t))$ as well as, for $\range i1n$, $\fo_i(\`x(t),\psi(\`x(t)))=0$, so $\fo_i(\`x(t),u(t))=0$, i.e.\ we have a solution of $\fo=0$.
  The converse is similar.
\end{proof}

The main result, proving $\fo=0$ is \Same, is harder because derivatives $\der{x_j}{k}$ move between rows in passing from $\fh$ to $\fo$.
It cannot be done purely structurally and needs Jacobian arguments, specifically the linear algebra relations that come from differentiating \rf{xtr0} using the chain rule.
Our proof method is to specify trial offsets for $\fo$ and show they (i) are \term{weakly valid} and (ii) give a nonsingular Jacobian. Then by 
\Cref{lm:valid_cd}, these offsets are  \term{valid}, proving S-amenability.

\begin{theorem}\label{th:extract}
With the above notation, if $\f=0$ is \Same, so is $\fo=0$.
\end{theorem}
\begin{remark}
  The reverse can be false: in fact extracting a subexpression that occurs several times and making it into a new variable is part of the toolkit of methods for converting an unamenable DAE to an equivalent amenable one, see \cite{Oki2021structmethods,tan2017conversion}.
\end{remark}

\begin{proof}
The replacement of $\psi(\`x)$ by $u$ is purely algebraic, so for $r\in R$ and each $j$, any derivative of $x_j$ that occurs in $\f_r$ must occur in either $\fo_r$ or $\phi$, or both.
Also no higher derivatives are created. Hence in terms of the \sigmxs $\Sigma=(\sij)$ of $\f$;\; $\Sigmah=(\sigmah_{ij})$ of $\fh$;\; and $\Sigmao=(\sigmao_{ij})$ of $\fo$:
\vspace{-1.5ex}
\begin{align}\label{eq:xtr5}
  \sij = \sigmah_{ij} = 
  \begin{cases*}
    \max(\sigmao_{ij},\sigmao_{\npp,j}), &if $i\in R$, \\
    \sigmao_{ij} &otherwise
  \end{cases*}
  \quad\text{for $i,\range j1n$}.
\end{align}
Let us divide the $(\npp)\x(\npp)$ extent of $\Sigmao$ into four parts:
\[ \small
\begin{tabular}{|c|c|}
\hline
 {\em body} \hfill        $1 \leq i \leq n$, $1 \leq j \leq n$ &
 {\em right edge} \hfill  $1 \leq i \leq n$, $j = \npp$ \\\hline
 {\em lower edge} \hfill  $i = \npp$, \>$1 \leq j \leq n$ &
 {\em corner} \hfill      $i = j = \npp$ \\\hline
\end{tabular}\;.
\]

(i)
Let $c_i,d_j$ be valid offsets for $\f$---that is, $c_i$ are for functions $\f_i$ ($\range i1n$) and $d_j$ are for variables $x_j$ ($\range j1n$) such that $d_j-c_i\ge \sij$ holds with equality on some transversal.
As trial values, we assign these same $c_i$ to $\fo_i$ ($\range i1n$), and $d_j$ to $x_j$ ($\range j1n$) considered as variables in $\fo$.

This gives weak validity in the body, because $d_j-c_i \ge \sij \ge \sigmao_{ij}$ by \rf{xtr5}.
We now choose $c_\npp$ for $\fo_\npp$ which is $\phi$, and $d_\npp$ for $x_\npp$ which is $u$.
To motivate the choice we look at some constraints.
In column $\npp$, belonging to $u$, we have
\vspace{-1ex}
\[ \sigmao_{i,\npp}=
  \begin{cases*}
     0     &for $i\in R$ (derivatives of $u$ do not occur, see below \rf{xtr0}), \\
     0     &also for $i=\npp$ by the definition of $\phi$, \\
     \ninf &otherwise. 
  \end{cases*}
\]
So weak validity for $j=\npp$ amounts to $d_\npp\ge c_i$ for $i\in R$ and also $d_\npp\ge c_\npp$.
On the basis of this choose trial values
\vspace{-1.5ex}
\begin{align}\label{eq:xtr5b}
  d_\npp &= \max_{r\in R} c_r, \quad c_\npp = d_\npp.
\end{align}
These give weak validity  on the right edge and corner, by construction.
There remains the lower edge.
For $\range j1n$ and $r\in R$, by validity for $\Sigma$, and \rf{xtr5},
\vspace{-1.5ex}
\begin{align*}
  d_j &\ge c_r + \sigma_{rj} \ge c_r + \sigmao_{\npp,j}.
\end{align*}
Take the maximum for $r\in R$: then \rf{xtr5b} gives $d_j \ge c_\npp + \sigmao_{\npp,j}$, i.e.\ weak validity on the lower edge.
This proves $c_1,\ldots,c_\npp$ and $d_1,\ldots,d_\npp$ are weakly valid for $\Sigmao$.

(ii)
Now, for a fixed $i\in R$, applying the chain rule to \rf{xtr4}(a) gives for $\range j1n$,
\vspace{-1.5ex}
\begin{align}
  \dbd{\fh_i}{\xcidj} &= \dbd{}{\xcidj} \fo_i(\`x,\psi(\`x)) 
    = \dbd{\fo_i}{\xcidj} + \dbd{\fo_i}{u} \dbd{\psi}{\xcidj} \notag\\
    &= \dbd{\fo_i}{\xcidj} - \dbd{\fo_i}{u} \dbd{\phi}{\xcidj}. \label{eq:xtr6}
\end{align}
This works also for $i\notin R$ (and in $\{1,\ldots,n\}$), because then $u$ does not appear in $\fo_i$ so $\tdbd{\fo_i}{u}=0$, and the second term vanishes---correct, because $\fh_i=\fo_i$ in this case.

Further, \rf{xtr6} works also for $j{=}\npp$, for which $x_j$ is $u$, for then 
\vspace{-1.5ex}
\[\text{$\xcd rj = \xcd r\npp = \der uk$ where $k=d_\npp-c_r\ge0$ by \rf{xtr5b}.}
\]
In the case $k>0$ we are differentiating wrt.\ $\dot{u},\ddot{u},\ldots$ which do not occur in any of the functions $\fh,\fo,\psi$, so both sides of \rf{xtr6} are $0$; while when $k=0$ we have on the left $\tdbd{\fh_r}{u} = 0$ since $u$ does not occur in $\fh$, and on the right
\vspace{-1.5ex}
\begin{align*}
  &\dbd{\fo_r}{u} - \dbd{\fo_r}{u} \dbd{\phi}{u} = \dbd{\fo_r}{u} - \dbd{\fo_r}{u} 1 = 0.
\end{align*}
Collecting left and right sides into row vectors, we have thus proved that for $\range i1n$,
{\nc\rvec[1]{\left(#1\right)_{j=1:\npp}}
\begin{align}
  \rvec{\dbd{\fh_i}{\xcidj}} &= \rvec{\dbd{\fo_i}{\xcidj}} - \dbd{\fo_i}{u} \cdot \rvec{\dbd{\phi}{\xcidj}}, \notag \\
  \~a &= \~b - \dbd{\fo_i}{u}\cdot\~c \quad\text{for short} . \label{eq:xtr8}
\end{align}
}\noindent
Here, items $\~a,\~b$ are the $i$th row of Jacobians $\Jach$ and $\Jaco$ respectively.
We now show the $\~c$ term is always a multiple (possibly zero) of $\Jaco$'s last row, so in an obvious notation,
\vspace{-1.5ex}
\begin{align}\label{eq:xtr9}
 \Jach(i,:) = \Jaco(i,:) - e_i\; \Jaco(\npp,:) \quad\text{for some scalar $e_i$, for $\range i1n$}.
\end{align}

There are three cases to consider:
\begin{compactitem}
  \item $i\notin R$. Then $u$ does not occur in $\fo_i$, so $\tdbd{\fo_i}{u}=0$, and we take $e_i=0$.
  \item $i\in R$ and $c_i\ne c_\npp$. By the definition \rf{xtr5b} of $c_\npp$ we have $c_i<c_\npp$.
  Thus
  \vspace{-1.5ex}
  \[ d_j-c_i > d_j-c_\npp \ge \sigmao_{\npp,j} \;, \]
  meaning $\xcidj$ is above the highest derivative of $x_j$ present in $\fo_\npp = \phi$, for $\range j1\npp$.
  Hence $\~c=0$ in \rf{xtr8}, and we take $e_i=0$.
  \item $i\in R$ and $c_i=c_\npp$. Then $\~c$ is precisely $\Jaco(\npp,:)$, and we take $e_i = \tdbd{\fo_i}{u}$.
\end{compactitem}
This proves \rf{xtr9}. Finally, clearly $\Jach(\npp,:) = \Jaco(\npp,:)$.
With \rf{xtr9}, this shows
\vspace{-1.5ex}
\begin{align}\label{eq:xtr10}
  \Jach &= E\;\Jaco, \quad
  \text{where $E=\bmx{\~I_n &-e \\0 &1}$ and $e$ is the column $n$-vector of $e_i$}.
\end{align}
We know $\det\Jach=\det\~J\ne0$.
Since $\det E=1$, eq.~\rf{xtr10} gives $\det\Jaco=\det\~J\ne0$.
Now \lmref{valid_cd} shows there exists an HVT of $\Jaco$ (it says nothing about where it is situated) on which $d_j-c_i=\sigmao_{ij}$ holds.

Thus, the $c_i,d_j$ are (strongly) valid offsets for the DAE $\fo(\`x,u)=0$ that give a nonsingular Jacobian.
By definition this DAE is \Same, as required.
\end{proof}

\subsection{Differentiating a DAE preserves S-amenability}\label{ss:diffdae}

We show the DAE that results from an \Same DAE by differentiating (with respect to $t$) any number of its equations, any number of times, is also \Same, and show how the solution set and \sysJ of the new and old systems are related.

First, using Griewank's Lemma (GL), see e.g.~\cite[Lemma 5.1]{nedialkov2007solving} for details, we prove this for one differentiation of one equation, and the main result \thref{diffdae} follows by induction.

To state GL, the meaning of differentiating a function $g$ of derivatives of variables $x_j$ needs clarifying.
We include $x_j$ in this terminology, as the $0$th derivative.
Suppose $g$ is defined as $t\ddot{x} + \d/\d t(xy^2)$; rewrite the second term to get $g = t\ddot{x} + \xp y^2 + 2xy\yp$.
A general $g$ can be expanded thus to be a function of selected $\der{x_j}{d}$, $\range j1n;\; d=0,1,\ldots$ seen as unrelated algebraic variables.
By the chain rule, $\dot{g}=\d g/\d t$ is a sum of terms
\vspace{-1ex}
\[
  \dbd{g}{\der{x_j}{d}} \ddt{\der{x_j}{d}} = \dbd{g}{\der{x_j}{d}} \der{x_j}{d+1}, 
  \quad\text{one for each occurring $\der{x_j}{d}$ in $g$}.
\]

GL says: if $\der{x_j}{d}$ is the leading (highest-order) derivative of $x_j$ in $g$, this is the {\em only} way the next higher derivative $\der{x_j}{d+1}$ occurs in $\dot{g}$, so the partial derivative of $\dot{g}$ wrt.\ $x_j$'s $(d{+}1)$th derivative equals the partial derivative of $g$ wrt.\ $x_j$'s $d$th derivative---generally false for non-leading derivatives.

In our example above, the leading derivatives are $\ddot{x}$ and $\yp$, so GL says $\tdbd{\dot{g}}{(\dddot{x},\ddot{y})}$ equals $\tdbd{g}{(\ddot{x},\yp)}$; indeed both equal $(t,2xy)$.

Some notation is useful.
First, we use DOF for the notion ``degrees of freedom'' and $\dof$ for their number: $\dof(\f)$ is the number of DOF of the DAE $\f=0$.
Second, given a DAE $\f=0$ of size $n$ and a pair of offset $n$-vectors $(c,d)$, as above, define for any $r$ in $1\:n$:
\begin{compactenum}[(i)]
  \item $\dif_r \f$ is the DAE that results from replacing $r$th function $\f_r$ by its $t$-derivative~$\fp_r$.
  \item $\dif_r(c,d)$ is the offset pair that results from the following two steps:\\
  -- Subtract 1 from $c_r$.\\
  -- If now $c_r=-1$, add 1 to all the $c_i$ and all the $d_j$.
\end{compactenum}

For the next lemma, let $\f$ be \Same with \sigmx $\Sigma$, and $\f^* = \dif_r \f$ with \sigmx $\Sigma^*$ for some $r$.
Let $(c,d)$ be a normalised, valid offset pair for $\f$, so Jacobian $\~J(\f;(c,d))$ is nonsingular.
Recall an HVT is a highest-value transversal.

\begin{lemma}\label{lm:diffdae}~
  \begin{enumerate}[(i)]
  \item $\Sigma^*$ and $\Sigma$ have the same set of HVTs.
  \item $(c^*,d^*) = \dif_r(c,d)$ is a normalised, valid offset pair for $\f^*$.
  \item $(c^*,d^*)$ gives the same Jacobian for $\f^*$ as does $(c,d)$ for $\f$:
  \vspace{-1.5ex}
  \begin{align}\label{eq:difrjac1}
  \~J\bigl(\f^*;(c^*,d^*)\bigr) = \~J\bigl(\f;(c,d)\bigr).
  \end{align}
  \item $\f^* = 0$ is \Same with $\dof(\f^*) = \dof(\f)+1$.\\
  Adding one suitable built-in IV gives $\f^*=0$ the same solution set as $\f=0$.
  \end{enumerate}
\end{lemma}
\begin{proof}
  (i,ii) The signature matrices $\Sigma=(\sij)$ of $\f$ and $\Sigma^*=(\sij^*)$ of $\dif_r \f$ are identical except in row $r$, where $\dif_r$ increases the order of derivative of each occurring variable by 1.
  Also, the construction of $(c^*,d^*)=\dif_r(c,d)$ makes $d^*_j-c^*_i$ equal $d_j-c_i$ except in row $r$ where it is increased by 1. Hence
  \vspace{-1.5ex}
  \begin{align}\label{eq:difrjac2}
    \sij^* =
    \begin{cases*}
      \sij+1 & \text{if $i=r$} \\
      \sij & \text{otherwise} \\
    \end{cases*}
    \quad\text{and}\quad
    d^*_j-c^*_i =
    \begin{cases*}
      d_j-c_i+1 & \text{if $i=r$} \\
      d_j-c_i & \text{otherwise} \\
    \end{cases*}
  \end{align}
  (for $\sigma_{rj}^*$ using $\ninf+1=\ninf$ for $j$ where $x_j$ does not occur in $\f_r$).

  That the pair $(c,d)$ is normalised means $\min_i c_i=0$; valid means $d_j-c_i\ge\sij$ everywhere with equality on some transversal $\calT$ of $\Sigma$, which is then necessarily an HVT of $\Sigma$, see \apref{smethod} \apitref{Sig4}.
  The construction of $(c^*,d^*)$ makes it normalised because $(c,d)$ is so; and by \rf{difrjac2} it is valid for $\Sigma^*$ because $(c,d)$ is valid for $\Sigma$, namely $d^*_j-c^*_i\ge\sij^*$ with equality on $\calT$.
  
  This argument also shows $\calT$ is an HVT of $\Sigma^*$.
  The same argument applies in reverse, showing that any HVT of $\Sigma^*$ is an HVT of $\Sigma$.
  Hence $\Sigma$ and $\Sigma^*$ have the same set of HVTs.
  
  (iii) Turning to Jacobians, from the definition \rf{sysJ}, left and right sides of \rf{difrjac1} are equal in all rows except the $r$th.
  We now prove equality in row $r$ also. For all $j$,
  \vspace{-1.5ex}
  \begin{align}\label{eq:difrjac4}
    d_j-c_r \ge \sigma_{rj},
  \end{align}
  implying $x_j^{(d_j-c_r)}$ is either the leading derivative of $x_j$ in $\f_r$ (if equality holds in \rf{difrjac4}), or does not occur in $\f_r$ at all (otherwise).
  In the former case, the $(r,j)$ entries in both Jacobians, namely
  \vspace{-1.5ex}
\[ \dbd{\f_r}{x_j^{(\sigma_{rj})}} \quad\text{and}\quad
     \dbd{\f_r^*}{x_j^{(\sigma_{rj}^*)}} = \dbd{\fp_r}{x_j^{(\sigma_{rj}+1)}} \]
  are equal by GL; in the latter, they are both zero and again equal.
  This proves \rf{difrjac1}.
  
  (iv) Hence $(c^*,d^*)$ are normalised valid offsets for $\dif_r \f$ that yield a nonsingular $\~J$.
  So by definition, the DAE $\dif_r \f=0$ is \Same.
  On $\calT$, the entries $\sij$ are all finite (not $\ninf$) and their sum is the DOF of the \Same DAE $\f=0$. 
  The corresponding $\sij^*$ equal $\sij$ except for the entry in row $r$, which is $\sij+1$.
  Hence $\dof(\dif_r \f) = \dof(\f) + 1$ as asserted.
  
  Equation $\fp_r=0$ that replaces $\f_r=0$ integrates to $\f_r = C = \text{const}$.
  So any built-in IV that makes $C=0$ ensures $\dif_r \f=0$ has the same solutions as $\f=0$.
\end{proof}

For the next theorem, we generalise $\dif$ to handle multiple differentiations.
Assume a DAE $\f=0$ of size $n$.
A {\em multi-index} is a tuple $k = (k_1,\ldots,k_n)$ with integer $k_r\ge0$.
We define $\dif^{(k)}\f$, where $k$ is a multi-index, to be the result of doing $\dif_r$ a total of $k_r$ times for $\range r1n$.
Differentiations on different components $\f_r$ of $\f$ are independent of each  other and can be done in any order, so it is clear that
\vspace{-1ex}
\begin{align}\label{eq:difrjac5}
  \dif^{(k)}(\dif^{(l)} \f) =\dif^{(k+l)} \f.
\end{align}

Correspondingly for an offset pair $(c,d)$, we define $\dif^{(k)}(c,d)$ as the result of two steps: first do $c := c-k$; then form $\gamma=\min_i c_i$ and do $c_:=c_i-\gamma$ and $d_j := d_j-\gamma$ for each $i,j$, so that the result has $\min_i c_i=0$.
It is not hard to see that as with \rf{difrjac5},
\vspace{-1ex}
\begin{align*}
  \dif^{(k)}\bigl(\dif^{(l)} (c,d)\bigr) =\dif^{(k+l)} (c,d).
\end{align*}

\begin{theorem}\label{th:diffdae}
  Let $\f$ be \Same with \sigmx $\Sigma$, and let $\f^* = \dif^{(k)} \f$ with \sigmx $\Sigma^*$ for some multi-index $k$.
  \begin{enumerate}[(i)]
  \item $\Sigma^*$ and $\Sigma$ have the same set of HVTs.
  \item $(c^*,d^*) = \dif^{(k)}(c,d)$ is a normalised, valid offset pair for $\f^*$.
  \item $(c^*,d^*)$ gives the same Jacobian for $\f^*$ as does $(c,d)$ for $\f$:
  \begin{align*}
  \~J(\f^*;(c^*,d^*)) = \~J(\f;(c,d)).
  \end{align*}
  \item $\f^* = 0$ is \Same with $\dof(\f^*) = \dof(\f) + |k|$, where $|k| = \sum_r k_r$.\\
      Adding $|k|$ suitable built-in IVs gives $\f^*=0$ the same solution set as $\f=0$.
  \end{enumerate}
\end{theorem}
\begin{proof}
  From \lmref{diffdae} by induction on the differentiations $\dif^{(k)}$ represents.
\end{proof}

\subsection{Inserting sub-ODEs preserves S-amenability}\label{ss:introsode}

Consider an \Same DAE \rf{daedef2} of size $n$,
\vspace{-1.5ex}
\begin{align*}
  \f(\`x)=0.
\end{align*}
Let $c_i,d_j$ be valid offsets for it, and $\calT$ be an HVT of its \sigmx $\Sigma$.

In \cl terms, let $\f$ contain an assignment $v=g(u)$, where $u$ is scalar and $v$ is an $m$-vector with $m\ge1$, so $g$ is an $m$-vector function.
In DAE terms, this assignment becomes an equation
\vspace{-1.5ex}
\begin{align}\label{eq:sode1}
  0 &=v-g(u), \quad\text{$u$ is scalar, $v$ is vector}
\end{align}
within the DAE, where $u$ and the components $v_i$ of $v$ are aliases for some $x_j$.

To simplify notation, let \rf{sode1} be the first $m$ equations of $\f$, i.e.\ we have
\vspace{-1.5ex}
\begin{align}\label{eq:sode2}
  0 = \f_i(\`x) &= v_i - g_i(u), \quad\range i1m.
\end{align}

In practical use, scenario \rf{sode2} arises where a standard function such as $\log$ occurs applied to some subexpression $\psi(\`x)$ within the expression defining the DAE in its original form.
The simplest case has $m=1$, and $g$ is this function; we convert $\psi(\`x)$ to a state variable $u$ (if it isn't already) and the output $g(u)$ to a state variable $v$ (if it isn't already), by two applications of \thref{extract}.
Case $m{>}1$ is needed \eg for $\cos$, which is part of a size 2 ODE that solves for $\cos$ and $\sin$ simultaneously.

Given the DAE as just described, we convert to a \sode in two stages.
First, differentiate the equations belonging to $g$.
In the notation of \thref{diffdae} this forms $\f^* = \dif^{(k)}\f$ where $k=(k_1,\ldots,k_n)$ has its first $m$ entries $1$, the rest $0$. By \rf{sode2} 
\vspace{-1.5ex}
\begin{align}\label{eq:sode3}
  \f^*_i(\`x) = 
  \begin{cases*}
    \fp_i(\`x) = \vp_i - g_i'(u) \up_i &for $\range i1m$, \\
    \f_i(\`x) &otherwise.
  \end{cases*}
\end{align}
Then substitute $h(u,v)$ in \rf{sode3}, where $g'(u)=h(u,g(u))$, to form $\fo$, where
\vspace{-1.5ex}
\begin{align}\label{eq:sode3b}
  \fo_i(\`x) = 
  \begin{cases*}
    \vp_i - h_i(u,v) \up_i &for $\range i1m$, \\
    \f_i(\`x) &otherwise.
  \end{cases*}
\end{align}

\begin{theorem}
  With the above notation,
  \begin{compactenum}[(i)]
  \item Every solution $x(t)$ of $\f=0$ is a solution of $\fo=0$; if $h$ is sufficiently smooth, the converse holds.
  \item Along such a solution, $\f=0$ and  $\fo=0$ have the same \sysJ in the sense of \thref{diffdae}.
  Hence $\fo=0$ is \Same.
  \end{compactenum}
\end{theorem}

\begin{example}\rm
The smoothness assumption is made to exclude cases like the following.
The (smooth) function $v=g(u)=u^3/27$ satisfies the non-smooth ODE $\d v/\d u = h(u,v) = v^{2/3}$.
Suppose we put $g$ in the ``Taylor library'', with \sode $\vp = h\up = v^{2/3}\up$, and $g$ as base-function defining the BIV.
Consider the ODE IVP
\vspace{-1.5ex}
\begin{align}\label{eq:insertex1}
  \yp &= t^3/27, \quad y(0)=0.  
\end{align}
The insertion process of this subsection (since $\dot{t}=1$ and the BIV is $v(t_0)=g(u(t_0))$ with $t_0=0$) converts \rf{insertex1} to the IVP
\vspace{-1.5ex}
\begin{align*}
  \vp = v^{2/3},\ \yp &= v;\quad v(0)=0,\ y(0)=0.
\end{align*}
More precisely, in the DAE notation of this subsection,
\vspace{-1ex}
\begin{align*}
  0=\f = \yp - \frac{t^3}{27} \text{ becomes }
  0 = \fo = \bmx{\, \fo_1\\ \fo_2} = \bmx{\vp - v^{2/3}\\ \yp - v}
  \text{ in variables } x = \bmx{x_1\\ x_2} = \bmx{v\\ y}.
\end{align*}
But besides $v(t)=t^3/27$, the $\fo_1$ equation has a solution $v(t)=0$.
Hence besides the correct $y(t)=t^4/108$, there is a solution $y(t)=0$.
(Indeed we can have $v=0$ up to an arbitrary $t=t^*>0$ and $v=(t-t^*)^3/27$ thereafter, with corresponding $y(t)$'s.)
\end{example}

\begin{proof}[Proof of Theorem]
(i) Going from $\f$ to $\f^*$ to $\fo$ just means the first $m$ equations go from \rf{sode2} to \rf{sode3} to \rf{sode3b}, the others being unchanged.
Clearly if $x(t)$ solves \rf{sode2} it solves \rf{sode3}, hence \rf{sode3b}, which proves the first assertion of (i).

The converse holds if an IVP for the ODE $\d v/\d u = h(u,v)$ never has more than one solution, which holds if, for instance, $h$ is Lipschitz \wrt $v$.
We omit the details.

For (ii), since $\f=0$ is \Same, we know by \thref{diffdae} that $\f^*=0$ is so; they have the same Jacobian, and the same solution set given suitable BIVs.

Let $\Sigma,\Sigma^*,\Sigmao$ be the \sigmxs of $\f,\f^*,\fo$ respectively (all $n\x n$).
For simplicity, we give pictures for $m=2$, since this sufficiently illustrates the general case.
The first $m$ rows of these matrices hold $\ninf$ except in columns belonging to $u$ and the $v_i$, as shown by these ``clips'':
\vspace{-1.5ex}
\begin{align}\label{eq:sode4}
  \Sigma &=
  \begin{blockarray}{c*3{@{\;\;}c}}
    &u &v_1 &v_2 \\\cline{2-4}
  \begin{block}{c|*3{@{\;\;}c}}
\f_1 &0 &0   &-   \\
\f_2 &0 &-   &0   \\
  \end{block}
  \end{blockarray},
  &
  \Sigma^* &=
  \begin{blockarray}{c*3{@{\;\;}c}}
    &u &v_1 &v_2 \\\cline{2-4}
  \begin{block}{c|*3{@{\;\;}c}}
\f^*_1 &1 &1   &-   \\
\f^*_2 &1 &-   &1   \\
  \end{block}
  \end{blockarray},
  &
  \Sigmao &=
  \begin{blockarray}{c*3{@{\;\;}c}}
    &u &v_1 &v_2 \\\cline{2-4}
  \begin{block}{c|*3{@{\;\;}c}}
\fo_1 &1 &1    &\le0 \\
\fo_2 &1 &\le0 &1    \\
  \end{block}
  \end{blockarray},
\end{align}
where as usual $-$ means $\ninf$, while $\le0$ means either $0$ or $\ninf$.

Let $\calT$ be an HVT of $\Sigma$. Then in rows $1$ to $m$ its entries are either all the zeros in the $v$ columns, or all but one of these and a zero in the $u$ column.
It is clear that the same $\calT$ is a transversal of $\Sigma^*$ and $\Sigmao$, and in fact an HVT of each, with
\vspace{-1ex}
\[  \Val(\Sigma)+m = \Val(\Sigma^*) = \Val(\Sigmao), \]
because this increase by $m$ is the largest achievable, and is achieved.

A set of valid offsets $(c,d)$ for $\Sigma^{*}$ is also valid for $\Sigmao$.
Using these for both $\f^*$ and $\fo$, formula \rf{sysJ} shows that in rows $1\:m$ (note that in \rf{sode4} the $1$s in the $\Sigma^*$ and $\Sigmao$ clips are leading derivative orders) we have
\vspace{-2ex}
\begin{align}\label{eq:sode5}
  \~J^*:&
  \begin{blockarray}{cccc}
    &u  &v_1      &v_2 \\\cline{2-4}
  \begin{block}{c|ccc}
\f_1 &-g_1'(u) &1  &0   \\
\f_2 &-g_2'(u) &0  &1   \\
  \end{block}
  \end{blockarray},
  &\Jaco:&
  \begin{blockarray}{cccc}
    &u  &v_1      &v_2      \\\cline{2-4}
  \begin{block}{c|ccc}
\ol\f_1 &-h_1(u,v) &1  &0   \\
\ol\f_2 &-h_2(u,v) &0  &1   \\
  \end{block}
  \end{blockarray}.
\end{align}
\vspace{-4ex}\par\noindent
On a solution $x(t)$ we have $v(t)=g(u(t))$ and $g'(u(t))=h(u(t),g(u(t)))$, so by \rf{sode5}, $\~J^{*}$ and $\ol{\~J}$ are equal in rows $1\:m$.
In rows $m{+}1$ onward, since function components and offsets are equal for $\f^*$ and $\fo$, $\~J^{*}$ and $\ol{\~J}$ are equal also.
Hence they are entry-wise equal, at any point on a solution path.
This shows $\fo=0$ is \Same.
\end{proof}

\section{Applying this to ODEs}\label{sc:applyODE}

\subsection{Offsets and Jacobian in the ODE case}

Let $C$ be a \cl of an ODE \rf{mainode} as in \ssref{codelists}.
To simplify notation, break vector assignments into components, so $C$ becomes a sequence of scalar assignments.
Using $\asg$ for the assignment operation, they are first the ODE-defining ones $\xp_i \asg \xh_i$, $i=1\:n$ (where $\xh_i$ is $x_{\out(i)}$); then $x_\npp \asg t$; then intermediate variables $x_j \asg \phi_j(x_{\prec j})$, $j=n{+}2\:N$.

To apply the previous theory, we must view $C$ as defining a DAE $G(t,x,\xp) = 0$ of size $N$, where $x$ is the whole vector $(x_1,\ldots,x_N)$.
A dual-purpose notation is useful:
\vspace{-1.5ex}
\[
  A:b\asg c \ \text{means}\ 
  \left\{\begin{tabular}{@{\,}l@{\,}l}
   either &$A$ is a label for assignment of expression $c$ to variable $b$, \\
   or     &$0=A:=b-c$, i.e.\ equation $A=0$, defining $A$ to be $b-c$.
  \end{tabular}\right.
\]
In this notation, the DAE $G=0$ can be written
\vspace{-1.5ex}
\begin{align}\label{eq:cl_DAE1}
  \begin{cases*}
    G_i: \xp_i \asg \xh_i,             &$i=1\:n$, \\
    G_i: x_\npp \asg t,                & $i=\npp$, \\
    G_i: x_i \asg \phi_i(x_{\prec i}), &$i=n{+}2\:N$.
  \end{cases*}
\end{align}
An $N$-vector function $x(t)$ solves this iff components $1\:n$ solve the original ODE \rf{mainode}.

Next put some, originally vector, assignments back to that form, so as to replace them by \sodes.
By \ssref{codelists}, we can assume the outputs of an $m$-vector assignment form a contiguous block, say $x_j$ for $j\in J = \{j^*{+}1,\ldots,j^*{+}m\}$. 
Necessarily, different assignments have disjoint sets $J$.

Consider one such assignment.
To be replaceable by a \sode, it has the generic form $v=g(u)$ where $u$ is scalar.
So $u$ is some $x_i$, and $v$ is $x_J$, meaning the vector of $x_j$ for $j\in J$.
Also ${\prec} j$, the set of $k$ such that $k\prec j$, must be the singleton set $\{i\}$ for each $j\in J$, equivalently vector $x_{\prec j}$ is the scalar $x_i$ for each $j\in J$.

We replace $v=g(u)$ by $\vp=h(u,v)\up$, where $h$ is an $m$-vector function satisfying \rf{huv}.
There might be several uses of the same $g$ and $h$ in the \cl, but components of this $h$ occurrence can be unambiguously labelled as $\chi_j$ for $j\in J$.
The $i$ belonging to a block $J$ needs an unambiguous label, so we call it $i(J)$.
Doing this for all the \sode blocks converts the DAE $G=0$ to $H=0$, where
\begin{align}\label{eq:cl_DAE2}
  \begin{cases*}
    H_j:\xp_j \asg \chi_j(x_i,x_J)\, \xp_i, &$j\in J$, $i=i(J)$, $J$ runs over the \sode blocks,\\
    H_j=G_j                                 &for $j$ outside any \sode block.
  \end{cases*}
\end{align}

Let $\Sigma=(\sij)$ be DAE \rf{cl_DAE2}'s $N\x N$ signature matrix, \apref{smethod} \apitref{Sig1}.
By the form of \rf{cl_DAE1,cl_DAE2} there is a transversal down the (main) diagonal of $\Sigma$, with
\vspace{-1.5ex}
\begin{align*}
  \sii = 
  \begin{cases*}
    1 &in an ODE-variable row $i=1\:n$, \\
    1 &in any \sode row, \\
    0 &in the other, call them ``ordinary'', rows.
  \end{cases*}
\end{align*}

Since an $x_i$ in an ordinary row depends only on earlier variables $x_{\prec i}$, there are no entries to the right of the diagonal except in the ODE and \sode rows.

In an ODE row, there is at most one off-diagonal entry, a zero in column $\out(i)$.
This can be right-of-diagonal but need not be, as $\out(i)$ might be $\le i$.
In a row, say $j$, of a \sode block \rf{cl_DAE2} (note $j$ here indexes rows, not columns) are the entries
\vspace{-1.5ex}
\begin{align*}
  \sij =
  \begin{cases*}
    1 &in column $i$ from $\xp_i$, always present, and left-of-diagonal since $i\prec J$, \\
    1 &on the diagonal from $\xp_j$, \\
    0 &potential $\sigma_{jk}$ values where $k\in J$, depending on the form of $\chi_j$.
  \end{cases*}
\end{align*}
Illustrated for $m=3$, the block looks like\vspace{-2ex}
{\small
\nc\TC[1]{\mbox{\textcircled{#1}}}
\begin{align}\label{eq:cl_DAE5}
\begin{blockarray}{rcccccccc}
    &&        &i &       &      & J   &     &       \\
\begin{block}{r\{c[ccc|ccc|c]}
    &&        &1 &       &\TC1  &\le0 &\le0 &       \\
 J  &&~\cdots &1 &\cdots &\le0  &\TC1 &\le0 &\cdots \\
    &&        &1 &       &\le0  &\le0 &\TC1 &       \\
\end{block}
\end{blockarray}
\end{align}
}\vspace{-4ex}

\noindent
where circled items are on the diagonal; an item ${\le}0$ is 0 or blank; columns marked $\cdots$ are blank.
Now define offset vectors $\~c=(c_1,\ldots,c_N)$, $\~d =(d_1,\ldots,d_N)$ for the DAE $0=H$ as follows.
\vspace{-1.5ex}
\begin{align}\label{eq:cl_DAE6}
  d_j &= 1 \quad\text{in all columns},
  &c_i &=
  \begin{cases*}
    0 & \text{in the ODE and \sode rows},\\
    1 & \text{in the ordinary rows}.
  \end{cases*}
\end{align}

\begin{theorem}\label{th:main1}~
  \begin{enumerate}[(i)]
  \item The offsets \rf{cl_DAE6} are valid and normalised.
  \item The main diagonal is an HVT of $\Sigma$.
  \item The \sysJ $\~J$ defined by these offsets is unit lower triangular, hence non-singular.
  \item Hence any DAE \rf{cl_DAE1} derived from an ODE, with an arbitrary set of \sode replacements \rf{cl_DAE2}, is \Same.
  \end{enumerate}
\end{theorem}
\begin{proof}
  (i) Normalised means $\min_i c_i = 0$, which is true because there is at least one ODE row.
  Valid means $d_j-c_i\ge\sij$ everywhere, with equality on some transversal.
  In ODE and \sode rows the inequality amounts to $\sij\le1$, which is true.
  In ordinary rows it amounts to $\sij\le0$, also true because no $1$'s occur in these rows.
  On the diagonal, which is a transversal, we have $\sii=0$ where $c_i=1$ and vice versa, so $d_j-c_i=\sij$ holds there, completing the proof of validity.
  
  Any transversal on which equality holds must be an HVT, so this also proves (ii).
  
  (iii) From the form of equations \rf{cl_DAE1,cl_DAE2} and formula \rf{sysJ} defining $\~J$, the latter equals $1$ on the diagonal.
  It remains to prove $\~J$ has no right-of-diagonal nonzeros.
  
  In an ordinary row this is true because $\Sigma$ has no entries there.
  In a \sode row we have $d_j-c_i = 1$, and the only right-of-diagonal $\sij$ are $\le0$, as illustrated in \rf{cl_DAE5}, so $\~J$ must be zero there.
  In an ODE row again $d_j-c_i = 1$. There is a $1$ on the diagonal, but all other $\sij\le0$ so again  $\~J$ is zero to the right of the diagonal.
  This proves $\~J$ is unit lower triangular, hence non-singular.
  
  (iv) A DAE is \Same if there exist valid offsets for it that give a nonsingular $\~J$.
  By (iii), this is the case and the proof is complete.
\end{proof}

\begin{example}\rm\label{ex:huvconcr2}
The example \rf{huvconcr2} in \ssref{subODEs} has\vspace{-2ex}
\begin{align*}
  &\left.\begin{aligned}
    \xp &= \xh\\
    u   &= -x \\
    \xhp &= \xh \up
  \end{aligned}\right\},
  &\Sigma &=
  \raisebox{-1ex}{$
  \begin{blockarray}{ccccc}
    &x &u &\xh &    \\
  \begin{block}{c[ccc]r}
    &0 &- &0 &\s0 \\
    &0 &0 &- &\s1 \\
    &- &1 &1 &\s0 \\
  \end{block}
   &\s1 &\s1 &\s1
  \end{blockarray}
  $}, &
  \~J &=
  \begin{blockarray}{ccccc}
    &x &u   &\xh & \\
  \begin{block}{c[ccc]c}
    &1 &0   &0   & \\
    &1 &1   &0   & \\
    &0 &-\xh&1   & \\
  \end{block}
  \end{blockarray}.
\end{align*}
\vspace{-5ex}

There is one ODE row, one \sode row and one ordinary row.
It can be seen that (i), (ii) and (iii) of the Theorem hold.
\end{example}

A DAE is an {\em implicit ODE} if it can be algebraically reduced to an ODE without differentiating any of the equations.
This is so iff it has valid offsets with all $c_i=0$.
\begin{corollary}
  The result of inserting \sodes in an ODE is an implicit ODE iff the input $u$ to each \sode occurrence is one of the state variables.
\end{corollary}
\begin{proof}
  Suppose the input to each \sode occurrence is one of the state variables.
  This is the same as saying that the reduction to the DAE $H=0$ \rf{cl_DAE2} did not introduce any intermediate variables, i.e.\ there are no ``ordinary rows'' in the \cl of $H$.
  Hence, all $c_i=0$ and we have an implicit ODE.
  The converse is left as an exercise.
\end{proof}

\subsection{Why the recurrence relations work}

We noted for the example in \ssref{subode_ok} that inserting the code for $h(u,v)$ of the \sode at the natural place in the \cl yielded a sound recurrence for Taylor coefficients, i.e.\ no TC value was required before it had been computed.

Now we show this is true in general---for any number of \sodes---because it is an instance of an \Same DAE's SSS.
We interpret the SSS definition, see \apref{smethod} \apitref{Sig8}, for \rf{cl_DAE2}.
Formula \rf{cl_DAE6} for its offsets gives $k_d=-1$, so that recurrence \rf{SSS} runs for $k=-1,0,\ldots,\kmax$.
Second, reverting to the assignment interpretation of \rf{cl_DAE2} shows that at each iteration of \rf{SSS} we can evaluate each of the $\der{x_j}{k+d_j}$ in order for $\range j1N$ (the fact that $\~J$ is lower triangular tells us the same).

Put values \rf{cl_DAE6} in \rf{SSS}, and use \rf{cl_DAE2}: we see that at initial stage $k=-1$ we~do

\vspace{-1ex}
\newpage
\begin{algo}{ODE using \sodes, stage $k=-1$}
{\footnotesize 1}\>for $\range i1N$ \\
{\footnotesize 2}\>\>solve nothing for $x_i$, if in an ODE row, $1\le i\le n$ \\
{\footnotesize 3}\>\>evaluate $G_i: x_i \asg \phi_i(x_{\prec i})$, if in an ordinary row \\
{\footnotesize 4}\>\>solve nothing for $x_i$, if in a \sode row
\end{algo}
Line 2 says: give a {\em user-supplied initial value} (IV) for each of $x_1,\ldots,x_n$.
Line 4 is less obvious, because \sodes are new and the SSS formalism has catching up to do.
It means: in a \sode block $J$, of size $m$, compute the $m$ {\em built-in initial values} (BIVs) that go with this block.
That is, evaluate $v=g(u)$ where
\begin{compactitem}
  \item input $u$ is an already known $x_i$, namely $x_{i(J)}$ in the notation of \rf{cl_DAE2};
  \item $g$ is the base function of this \sode;
  \item output $v$ is $x_J$, the $m$ variables belonging to this block.
\end{compactitem}
After this we have values of all the variables $x_j$, equivalently their order $0$ \tcs.
For higher stages, note the SSS is phrased in terms of derivatives, but in practice one computes \tcs, which are scaled derivatives \rf{TCdef}.
Next, stage $k=0$ does
\begin{algo}{ODE using \sodes, stage $k=0$}
{\footnotesize 1}\>for $\range i1N$ \\
{\footnotesize 2}\>\>set $\xp_i=\xh_i$, if in an ODE row, $1\le i\le n$ \\
{\footnotesize 3}\>\>evaluate $\dot{G}_i: \xp_i \asg \ddt{}\phi_i(x_{\prec i})$, if in an ordinary row \\
{\footnotesize 4}\>\>evaluate a component of $\vp=h(u,v)\up$, if in a \sode row \\
\>\>\>(where $\up = \xp_{i(J)}$ has been set, and vector $\xp_J = \vp$ is to be set)
\end{algo}
After this we have values of all first derivatives of the $x_j$, equivalently the order $1$ \tcs.
Clearly this pattern continues, with stage $k$ computing the order $(k+1)$ \tcs of all the variables.
Note the $k$th derivative of $\phi_i(x_{\prec i})$ in general involves not only the $k$th derivative of  $x_{\prec i}$ but all its lower derivatives as well, and similarly with $\vp=h(u,v)\up$.

The combination of base function $v=g(u)$ at Taylor order $0$, with $\vp=h(u,v)\up$ or its derivatives at higher orders, is exactly what the \sode operator $\DDOTi$ does in \dfref{sodedef}.
E.g., \exref{huvconcr2} comes from the example in \ssref{subODEs} that gives recurrence \rf{huvconcr4}, and it can be seen this agrees with the two algorithms above.

\subsection{Example \cl: a nonlinear spring-pendulum}\label{ss:springpend}
{
\nc\Vs{V_\text{spr}}
\nc\Vg{V_\text{grav}}
\rnc\th{\theta}
\nc\rdot{\dot{r}}
\nc\thdot{\dot{\theta}}
\nc\rDot{s}
\nc\thDot{\omega}
We give here a 2D mechanical system model leading to an ODE \rf{spendnl2}.
First, its modest-sized \cl, see \apref{springpend}, illustrates a scalar and a vector \sode, for $\exp$ and $(\cos,\sin)$ respectively; second, it is used for numerical tests in \scref{results}.

It is a massless linear spring, of equilibrium length $a$, hung at one end from a fixed point $O$, with a point mass $m$ attached at the other end $P$, moving under gravity $g$.
The spring is nonlinear---its potential energy $\Vs$ is not $\frac12 k(r-a)^2$ as with a classic Hooke spring of modulus $k$, but has an extra term, namely
\vspace{-1.5ex}
\[ \Vs = k\dtimes v(r-a), \quad\text{with $k>0$ constant, $v(x) = \tfrac12 x^2 + e^{-x}-1+x$}. \]
Like $\tfrac12 x^2$, the extra part $e^{-x}-1+x$ is a convex function with minimum $0$ at $x=0$, hence so is their sum.
As $r$ decreases from $a$ the spring stiffness $\d^{\;2}\Vs/\d r^2$ increases rapidly; as $r$ increases past $a$ it approaches linear stiffness $k$. The pendulum ``dislikes being compressed too much''.
Position coordinates are polars, $q{=}(q_1,q_2){=}(r,\th)$ where:
\begin{quote}
  $r$ is length $OP$, \\
  $\th$ is angle of $OP$ anticlockwise from downward vertical.
\end{quote}
Using these we form Lagrangian $L=T-V$ where
 $T$ is kinetic energy $\frac12 m(\rdot^2 + (r\thdot)^2)$, and 
 $V$ is potential energy, the sum of $\Vs$ as above and gravity energy $\Vg = -mg\cos\th$.
The Euler--Lagrange equations $\ddt{}(\tdbd{L}{\dot{q}_i}) = \tdbd{L}{q_i}$ give the equations of motion as an unconstrained system of two second-order differential equations
\begin{align}
  \tddt{}(m\rdot)     &= mr\thdot^2 + mg\cos\th - k\dtimes\bigl((r-a) + 1 - e^{-(r-a)}\bigr), \label{eq:spendnl1a}\\
  \tddt{}(mr^2\thdot) &= -mgr\sin\th. \label{eq:spendnl1b}
\end{align}
Dividing \rf{spendnl1a} by $m$ and \rf{spendnl1b} by $mr$, and defining $\rDot=\rdot$ and $\thDot=\thdot$, reduces this to a first-order ODE system in $(r,\rDot,\th,\thDot)$:
\vspace{-1.5ex}
\begin{align}\label{eq:spendnl2}
\left.\begin{aligned}
  \rdot       &= \rDot, &
  \dot{\rDot} &= r \thDot^2 + g \cos\th - \tfrac{k}{m} \bigl((r-a) + 1 - e^{-(r-a)}\bigr) \\
  \thdot      &= \thDot, &
  \dot{\thDot}&= (-g\sin\th - 2\rDot\thDot)/r
\end{aligned}\right\}.
\end{align}
}

\section{Numerical Results}\label{sc:results}

\nc\tol{\tt tol}
\nc\atol{\tt atol}
\nc\rtol{\tt rtol}
\nc\scd{\text{SCD}\xspace}
\nc\mc\multicolumn
\nc\figscale{.30}
\newcommand{\tableinp}[2]
{%
\footnotesize
\renewcommand{\arraystretch}{1.2}
\setlength{\tabcolsep}{5pt} 
\renewcommand{\times}{\dtimes}
\centering
{
\rnc\tsc{5.5pt}
\begin{tabular}{l@{\hskip \tsc} l @{\hskip \tsc} r@{\hskip \tsc}  r @{\hskip \tsc}r @{\hskip \tsc}l}
\mc{5}{c}{#2}\\
\mc{1}{c}{Solver}  & \mc{1}{c}{Tol} & 
\mc{1}{c}{\scd} & \mc{1}{c}{Steps} & \mc{1}{c}{Time (s)}\\ \hline
\input #1
\end{tabular}%
}
}
We have implemented the above theory in \matlab.  The \cl generation is done through operator overloading by executing the function defining the ODE. Then our ODE solver, named \odets, ODE by Taylor Series, takes the generated \cl as input and evaluates it on each integration step to compute TCs. This is a fixed-order, variable-stepsize integrator. 

We show the performance of \odets on three problems, using \matlab R2024b on an Apple M3 MacBook Air 2024 with 16GB main memory. 

\begin{enumerate}[A.]
\item \label{it:sp} The nonlinear spring-pendulum of \ssref{springpend}, 
\item \label{it:plei} The Pleiades problem from the test set \cite{TestSetIVP}, and
\item  \label{it:bruss} The Brusselator system used in 
\matlab's stiff solver documentation \cite{mathworks_stiff_odes}.
\end{enumerate}
\rnc{\itref}[1]{\ref{it:#1}}

The Pleiades ODE models the motion of $7$ stars under mutual gravitational attraction (simplified from $3$ dimensions to $2$); it has size 28 when reduced to first~order.
The Brusselator models diffusion in a chemical reaction; it has size $2N$ where $N$ is the number of intervals in the space grid, and is increasingly stiff and sparse as $N$ increases.
For the non-stiff problems \itref{sp} and \itref{plei}, we compare the performance of \odets  with that of the \matlab solvers \li{ode113}, \li{ode45}, and \li{ode89}.
(Results for \li{ode78} are very like those for \li{ode89}.)
In \odets, the order is set (following a formula of \cite{jorba2005software}) as
\vspace{-1.5ex}
\begin{align}\label{eq:orderchoice}
  p  = \lceil -0.5 \log(\min(\atol, \rtol)) + 1\rceil,
\end{align}
\vspace{-4.5ex}\par\noindent
where  $\atol$ and $\rtol$ are absolute and relative tolerances, respectively. 

For problem \itref{bruss}, we compare with \matlab's stiff solver \odestiff.
For stiff problems, \rf{orderchoice} seems far from optimal for \odets, and its order was fixed at $p=20$.
Barrio \cite{Barrio05b} found experimentally that the stability region of the order $p$ explicit TS method approaches a semicircle $|z| \le r(p)$, $\Re(z)\le 0$ where $r(p)$ linearly increases with $p$.
Hence for a stiff ODE, where stepsize is typically limited by stability, not accuracy (provided the dominant eigenvalue $\lambda$ is not too close to the imaginary axis), the stepsize one can take also increases linearly, being roughly $h=r(p)/|\lambda|$.

A reference solution at the end of the integration interval is computed with \odeen (for \itref{sp} and \itref{plei}) and with \odestiff (for \itref{bruss}) with absolute and relative tolerances of $3\dtimes 10^{-14}$.  
We then integrate for  tolerances (same absolute and relative)  $10^{-3}, \ldots, 10^{-13}$. We define  the number of significant correct digits (\scd) at $\tend$ by
\vspace{-1.5ex}
\[
  \scd = -\log_{10} \left(\norm{\text{relative error at $\tend$}}_\infty \right),
\vspace{-1.5ex}
\]
and draw {\em work-precision diagrams} of CPU time versus \scd at various tolerances.

In the Figures, the ``odets'' plot represents the integration CPU time only, excluding \cl time, which is depicted by the ``CL'' line.
In the Tables, the “steps” column gives the number of accepted (A) and failed (F) steps as~A/F.

\begin{figure}[ht]
\renewcommand{\thesubfigure}{\Alph{subfigure}.}
\centering
\hfill
\subfigure[Spring-pendulum]{\includegraphics[scale=\figscale]{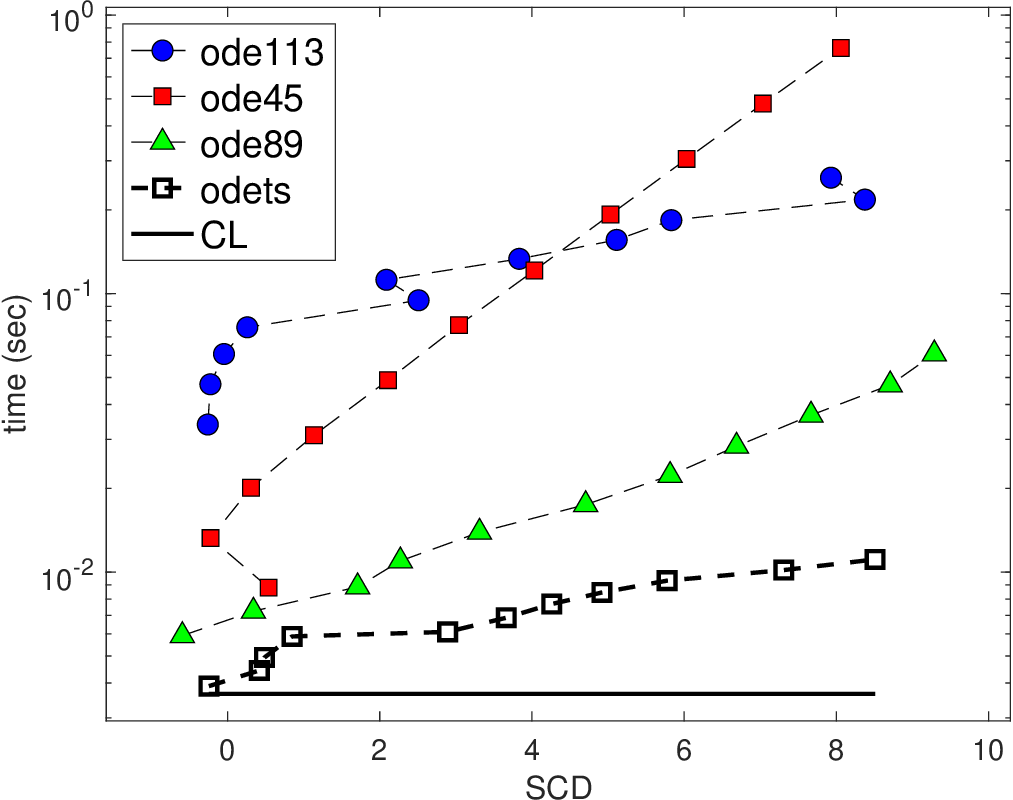}} 
\hfill
\subfigure[Pleiades]{\includegraphics[scale=\figscale]{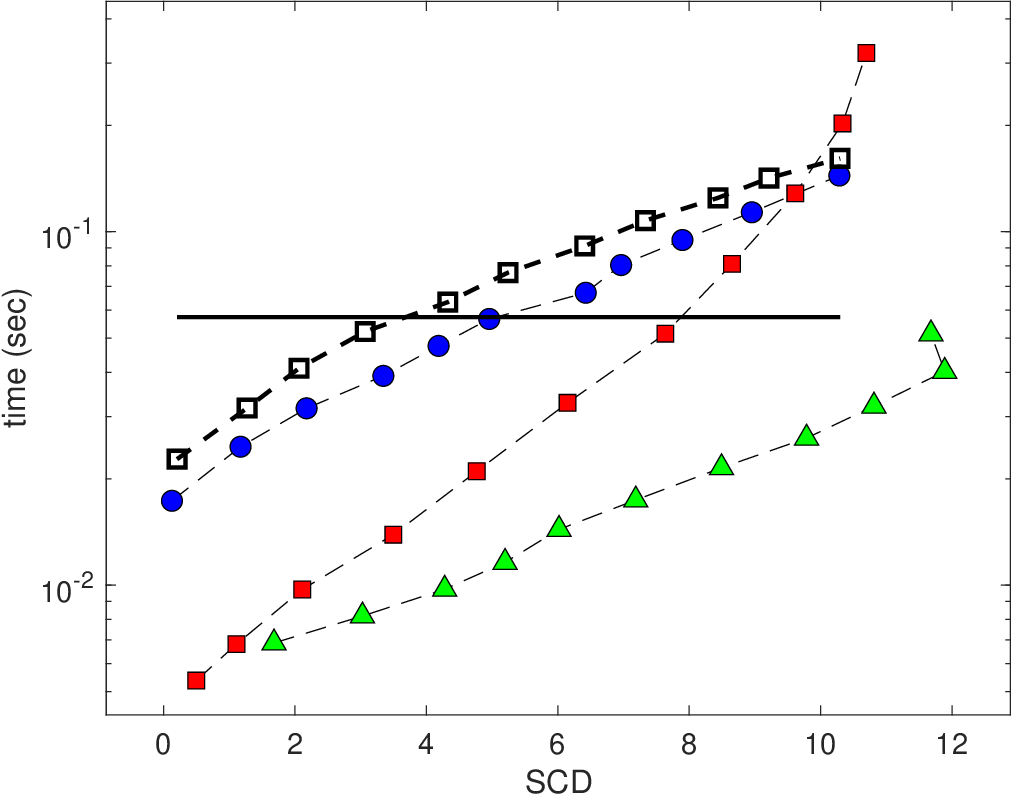}}
\hfill
\caption{Work-precision diagrams for Problems A and B. \label{fg:plei}}
\end{figure}

\begin{table}[ht]
\caption{Performance results for Problems A and B. \label{tb:plei}}
\centering
\nc\tsc{5.5pt}
\footnotesize
\renewcommand{\arraystretch}{1.2}
\setlength{\tabcolsep}{5pt} 
\renewcommand{\times}{\dtimes}
\centering
{
\rnc\tsc{5.5pt}
\begin{tabular}{l@{\hskip \tsc} l @{\hskip \tsc} r@{\hskip \tsc}  r @{\hskip \tsc}r @{\hskip \tsc}l}
\mc{5}{c}{Spring-pendulum}\\
\mc{1}{c}{Solver}  & \mc{1}{c}{Tol} & 
\mc{1}{c}{\scd} & \mc{1}{c}{Steps} & \mc{1}{c}{Time (s)}\\ \hline
\multirow{5}{*}{\li{odets} } & $10^{-5}$ & $0.49$ & 431/14 & \num{4.9e-03}\\ 
& $10^{-7}$ & $2.89$ & 414/13 & \num{6.1e-03}\\ 
& $10^{-9}$ & $4.26$ & 468/20 & \num{7.7e-03}\\ 
& $10^{-11}$ & $5.78$ & 504/17 & \num{9.3e-03}\\ 
& $10^{-13}$ & $8.51$ & 538/12 & \num{1.1e-02}\\ 
\hline 
\multirow{5}{*}{\li{ode89}}  & $10^{-5}$ & $1.71$ & 174/12 & \num{8.9e-03}\\ 
& $10^{-7}$ & $3.31$ & 275/15 & \num{1.4e-02}\\ 
& $10^{-9}$ & $5.82$ & 446/11 & \num{2.2e-02}\\ 
& $10^{-11}$ & $7.67$ & 737/ 9 & \num{3.7e-02}\\ 
& $10^{-13}$ & $9.29$ & 1228/ 5 & \num{6.1e-02}\\ 

\end{tabular}%
}
\hfill
\footnotesize
\renewcommand{\arraystretch}{1.2}
\setlength{\tabcolsep}{5pt} 
\renewcommand{\times}{\dtimes}
\centering
{
\rnc\tsc{5.5pt}
\begin{tabular}{l@{\hskip \tsc} l @{\hskip \tsc} r@{\hskip \tsc}  r @{\hskip \tsc}r @{\hskip \tsc}l}
\mc{5}{c}{Pleiades}\\
\mc{1}{c}{Solver}  & \mc{1}{c}{Tol} & 
\mc{1}{c}{\scd} & \mc{1}{c}{Steps} & \mc{1}{c}{Time (s)}\\ \hline
\multirow{5}{*}{\li{odets} } & $10^{-5}$ & $2.06$ & 185/ 9 & \num{4.1e-02}\\ 
& $10^{-7}$ & $4.33$ & 185/23 & \num{6.3e-02}\\ 
& $10^{-9}$ & $6.41$ & 209/23 & \num{9.1e-02}\\ 
& $10^{-11}$ & $8.44$ & 229/23 & \num{1.2e-01}\\ 
& $10^{-13}$ & $10.30$ & 245/19 & \num{1.6e-01}\\ 
\hline 
\multirow{5}{*}{\li{ode89}}  & $10^{-5}$ & $4.28$ & 74/19 & \num{9.7e-03}\\ 
& $10^{-7}$ & $6.02$ & 109/29 & \num{1.4e-02}\\ 
& $10^{-9}$ & $8.49$ & 169/33 & \num{2.1e-02}\\ 
& $10^{-11}$ & $10.81$ & 267/13 & \num{3.2e-02}\\ 
& $10^{-13}$ & $11.68$ & 443/ 1 & \num{5.1e-02}\\ 

\end{tabular}%
}

\end{table}

\subsection{Spring-pendulum and Pleiades}
For these, work-precision diagrams are in \fgref{plei}.
Details for \odets and \odeen are in \tbref{plei}.

\paragraph{Problem A}
We use  $g = 9.81, k = 40, m = 1, a = 1$ and initial values
\vspace{-1ex}
\[ \bigl(r(0), s(0), \theta(0), \omega(0)\bigr)=(r_0, 0, \pi/4, 4.65) \]
where $r_0 = a + m g/k$ is the pendulum length at rest. We integrate for $t\in [0,20]$. 

\paragraph{Problem B}
The name is from the Pleiades constellation (``Seven Sisters''), known to be a genuine cluster of seven close stars.
The equations are for $n=7$ point-bodies moving under gravitational attraction; if body $B_i$ has mass $m_i$ and position vector $p_i$, the force exerted on $B_i$ by $B_j$ ($j\ne i$) is
\vspace{-1ex}
\[
  F_{ij} = G \frac{m_i\, m_j}{\nrm{p_i-p_j}^2} \unitv(p_j-p_i)
  = G \frac{m_i\, m_j. (p_j-p_i)}{\nrm{p_i-p_j}^{3/2}},
\]
where $G$ is the gravitational constant, $|x|$ is the Euclidean length of vector $x$, and $\unitv(x) = x/\nrm{x}$ is the unit vector in the direction of $x$.
Newton's second law gives the equations of motion
\[ m_i\, \ddot{p}_i = \sum_{j, j\ne i} F_{ij} \quad\text{for $\range i1n$}. \]
The model is in 2D, so using cartesian coordinates $p_i=(x_i,y_i)$ it reduces to $4n=28$ first-order equations.
See \cite{TestSetIVP} for details, and the parameter values and initial conditions used.
The motion is notable for several very close encounters (quasi-collisions).

\begin{figure}[ht]
\rnc\figscale{0.24}
\centering
\hfill
\subfigure[$N=20$]{\includegraphics[scale=\figscale]{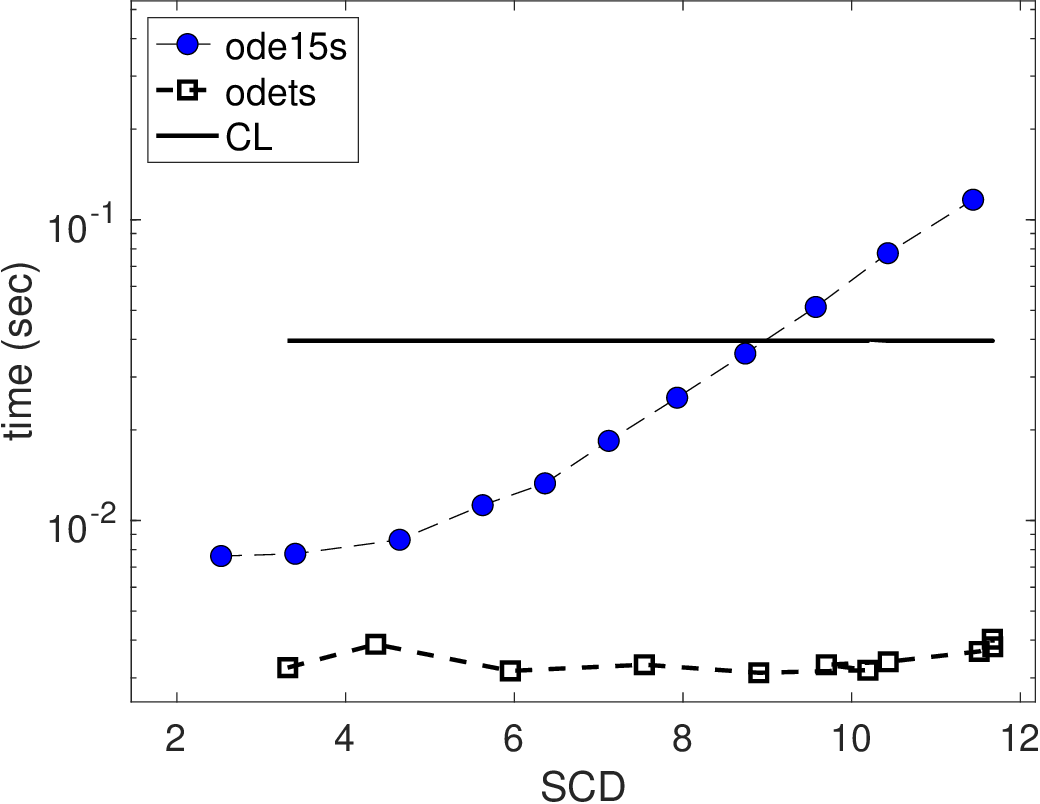}}\hfill
\subfigure[$N=40$]{\includegraphics[scale=\figscale]{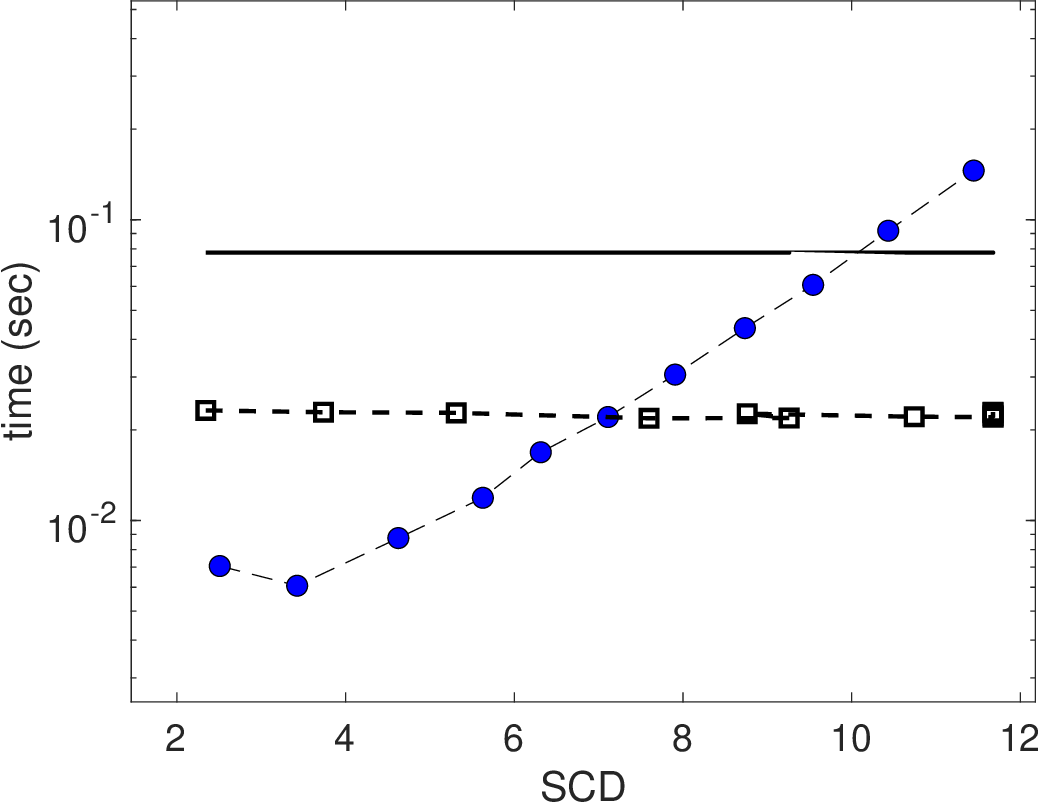}}\hfill
\subfigure[$N=100$]{\includegraphics[scale=\figscale]{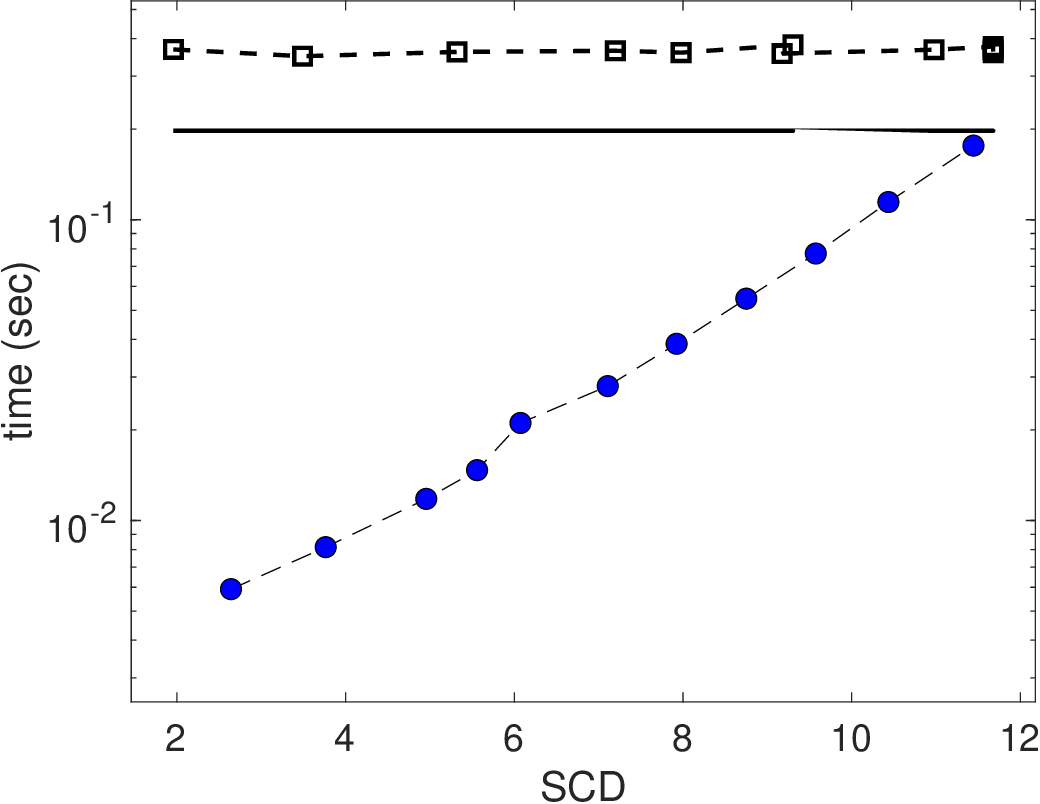}} \hfill
\caption{Work-precision diagrams for Problem C. \label{fg:brus}}
\end{figure}

\begin{table}[ht]
\caption{Performance results for Problem C. \label{tb:brus}}
\centering
\footnotesize
\rnc\times\dtimes
\renewcommand{\arraystretch}{1.2}
\setlength{\tabcolsep}{3.8pt} 

\begin{tabular}{l |r r c |r r c |r r c  c}
\multicolumn{10}{c}{\odets}\\ 
\multicolumn{1}{c}{}
  & \multicolumn{3}{c}{$N=20$}
  & \multicolumn{3}{c}{$N=40$}
  & \multicolumn{3}{c}{$N=100$}
\\ \hline
\multicolumn{1}{c|}{Tol} 
  & \scd & Steps & Time (s) 
  & \scd & Steps & Time (s) 
  & \scd & Steps & Time (s) 
\\  [2pt]
 $10^{-5}$
  & $5.95$ & 41/ 3 & \num{3.2e-03} 
  & $5.31$ & 153/ 3 & \num{2.3e-02} 
  & $5.32$ & 920/29 & \num{3.6e-01} \\ 
 $10^{-7}$
  & $8.90$ & 42/ 3 & \num{3.1e-03} 
  & $9.26$ & 152/ 8 & \num{2.2e-02} 
  & $7.98$ & 920/29 & \num{3.6e-01} \\ 
 $10^{-9}$
  & $9.70$ & 45/ 5 & \num{3.3e-03} 
  & $10.74$ & 153/ 6 & \num{2.2e-02} 
  & $9.17$ & 921/31 & \num{3.6e-01} \\ 
 $10^{-11}$
  & $11.51$ & 49/ 7 & \num{3.7e-03} 
  & $11.68$ & 154/ 4 & \num{2.2e-02} 
  & $11.68$ & 921/31 & \num{3.8e-01} \\ 
 $10^{-13}$
  & $11.67$ & 54/ 4 & \num{4.0e-03} 
  & $11.68$ & 155/ 4 & \num{2.3e-02} 
  & $11.68$ & 922/26 & \num{3.6e-01} \\ 
[2pt]
\multicolumn{10}{c}{\li{ode15s}} \\ 
   $10^{-5}$
  & $4.64$ & 158/ 8 & \num{8.6e-03} 
  & $4.63$ & 161/10 & \num{8.7e-03} 
  & $4.96$ & 168/ 9 & \num{1.2e-02} \\ 
 $10^{-7}$
  & $6.36$ & 318/11 & \num{1.3e-02} 
  & $6.31$ & 324/ 9 & \num{1.7e-02} 
  & $6.07$ & 330/ 9 & \num{2.1e-02} \\ 
 $10^{-9}$
  & $7.93$ & 646/12 & \num{2.6e-02} 
  & $7.91$ & 659/11 & \num{3.1e-02} 
  & $7.92$ & 671/10 & \num{3.9e-02} \\ 
 $10^{-11}$
  & $9.57$ & 1351/12 & \num{5.1e-02} 
  & $9.54$ & 1376/12 & \num{6.1e-02} 
  & $9.57$ & 1404/13 & \num{7.7e-02} \\ 
 $10^{-13}$
  & $11.44$ & 2862/13 & \num{1.2e-01} 
  & $11.45$ & 2916/13 & \num{1.5e-01} 
  & $11.44$ & 2964/13 & \num{1.8e-01} \\ 
\end{tabular}

\end{table}

\subsection{Brusselator}
The equations are 
\vspace{-1ex}
\begin{align*}
\begin{split}
\dot u_i &= 1 + u_i^2 v_i - 4u_i + \alpha (N+1)^2 (u_{i-1} - 2i + u_{i+1})\\
\dot v_i &= 3u_i - u_i^2 v_i + \alpha (N+1)^2 (v_{i-1} - 2v_i + v_{i+1}),\qquad \range i1N.
\end{split}
\end{align*}
As in \cite{mathworks_stiff_odes}, we set $\alpha = 1/50$ and solve over time interval $[0, 10]$ with initial conditions
\vspace{-1.5ex}
\begin{align*}
u_i(0) &= 1 + \sin\left(2\pi \tfrac{i}{N+1}\right), \qquad
v_i(0) = 3, \quad \range i1N.
\end{align*}
We integrate for $N=20$, $40$ and $100$.
The \li{brussode.m} example code \cite{mathworks_stiff_odes} uses \odestiff and exploits sparsity, which  \odets does not because it does not use linear algebra. We augment \li{brussode.m} by adding a call to \odets, and compare both solvers.  

For these values of $N$, we consider the corresponding systems as mildly stiff, as the explicit TS method of \odets does well. 
Its stepsize is primarily constrained by stability, resulting in, at each $N$ value, a similar number of steps for all the tolerances (see discussion below \rf{orderchoice}, and \tbref{brus}).
Consequently, CPU time for \odets is nearly the same at all tolerances.
Remarkably, the integration part of \odets, excluding \cl time, outperforms \odestiff when $N{=}20$, and partially when $N{=}40$.

\clearpage
\section{Conclusion}\label{sc:concl}

We believe the results on \Same DAEs presented here are valuable in their own right and will prove useful in other ways.
For us here, however, they are mainly a way to show \odets is theoretically well founded.
Here, we sketch some implementation issues, including those raised by \sodes.

\subsection{From theory to practice}

\subsubsection{Parameters}
``Code-list time'' converts function $f(t,y)$ to a \cl object $C$ that is executed at ``run time'' to integrate an IVP.
The process can be expensive (see the CL line in the Figures in \scref{results}), so we want $C$ to be usable for multiple integrations.
Varying initial conditions is not a difficulty, as they are supplied to a solver independently of $f$ or $C$.
But values such as $G$ and the $m_i$ in the Pleiades system present a difficulty: by default, they become ``immediate'' numerical values in \cl lines, with no provision to modify them.
So---besides {\em variables}, that at run time are \TS, and {\em constants} that become immediate values---\odets has {\em parameters} (``variable constants'').
They stay constant during one integration, and can be changed cheaply at ``parameter time'' between integrations.

The $N$ in the Brusselator cannot be handled this way, since a change in $N$ changes the structure of the \cl, not just some numbers in it.

\subsubsection{The power function}
The function $x^y$ is messy even for numeric $x,y$, conflating several mathematical definitions.
This worsens when $x$ or $y$ become functions of $t$.
Here we only consider the case $v=u^c$ where $u$ is a variable, $c$ is not.

The \sode definition based on $\d v/\d u = cv/u$ is simple, but refers to a mathematical $u^c$ that in general is a multi-valued function in the complex plane.
Because of the division, it inherently cannot handle $u{=}0$, but in a common case, $u{=}0$ is quite ordinary: where $c$ is a non-negative integer.
Here $u^c$ is defined for all $u$ and can be~found by repeated square and multiply, e.g.\ $u^{11} = \left((u^2)^2\dtimes u\right)^2\dtimes u$ with 5 multiplications.

For constant $c$, one knows at \cl time which of these two functions applies, and can hardwire the right code.
For parameter $c$ one cannot do this: changing $c$ between integrations can change the \cl's structure, not just the numbers in it---as with $N$ in the Brusselator.
As all ways to resolve this need extra machinery, we opt for simplicity and currently do not support $u^c$ where $c$ is a parameter.

\subsubsection{Removing duplication}
In \sode terms $\cos()$ and $\sin()$ are components of function $\cs()$, see \exref{subodes}.
If, say, $\cos(x_7)$ and $\sin(x_7)$ occur in the $f(t,y)$ definition, one wants one instance of $\text{cs}(x_7)$ in the \cl , not two.
To achieve this, we use a dictionary, whose {\em key} is a tuple (operation, actual arguments) and the corresponding {\em value} is the \cl line that does this operation on these arguments.
If the same tuple reoccurs, we re-use that line instead of making a new one.

This is a common-subexpression eliminator, and is effective because it nips such expressions in the bud rather than searching for them after they have been made.
For efficiency the tuple is hashed to a 64-bit integer by \matlab's \li{keyHash} function.

\subsection{Future work}
 We plan to release the \odets software, with an article that describes its design, implementation, and interface, and thorough numerical studies. 
 
Reducing the computational kernel to only $5$ operations, the \baos implemented by \rf{arithops} and $\DDOT{}$ by \rf{ddot1}, facilitates computing partial derivatives of TCs by AD.  
Preliminary trials show finding gradients using forward AD through this kernel is straightforward. 
This opens at least two directions: implement an implicit ODE solver based on Hermite-Obreschkoff formulas \cite{corless2024hermite,mazzia2018class,zolfaghari2023hermite}; and compute sensitivities \wrt initial conditions. Experiments with the former have found it a viable approach. 


\newpage
\appendix

\section{Structural amenability and the Signature method}\label{ap:smethod}~
\nc\calF{\mathcal{F}}
\nc\calX{\mathcal{X}}

\subsection*{Context}
For details and proofs of what is outlined here, see \cite{Pryce2001a}.
Roughly stated, a DAE is \Same if structural analysis (SA) of its sparsity \term{succeeds}: then it tells you how to solve the DAE numerically.
SA in this sense (electric circuit theory has a complementary kind of SA) began with the Pantelides algorithm \cite{Pant88b}, shown in \cite{Pryce2001a} to be equivalent to the \Smethod.

Pantelides's method is built into environments such as Modelica, gProms and EMSO \cite{openmodelica, gpromsUG, EMSO2010}, so any use of these popular systems tacitly assumes S-amenability.

\subsection*{The method}
The \Smethod applies to a DAE \rf{maindae}, of size $n$.
It has the following steps and basic ideas.

\begin{enumerate}[$\Sigma$1.]
\item\label{it:Sig1}
Form the $n\x n$ {\em signature matrix} $\Sigma=(\sij)$ where $\sij$ is the highest order of derivative of $x_j$ occurring in $\f_i$, or $\ninf$ if $x_j$ does not occur in $\f_i$.

\item\label{it:Sig2}
 A \term{transversal} is a set $\calT$ of $n$ positions $(i,j)$ in $\Sigma$ with exactly one in each row and each column.
It can be enumerated e.g.\ as $\calT= \{(1,j_1),(2,j_2),\dots,(n,j_n)\}$ where $(j_1,\dots,j_n)$ is a permutation of $1\:n$.
Its \term{value}, a finite integer or $\ninf$, is
\[\Val(\calT) :=\sum_{(i,j)\in\calT}\sij.\]

\item\label{it:Sig3}
 If a $\calT$ of finite value exists---i.e.\ no $\sij$ on it is $\ninf$---the DAE is \term{structurally well-posed} (SWP).
It then has at least one \term{highest-value transversal} (HVT), a $\calT$ of largest value, and we define the (finite) \term{value of $\Sigma$} to be $\Val(\Sigma) = \Val(\calT)$.

\item\label{it:Sig4}
 A pair of \term{offsets} is two integer vectors $c=(c_i,\ldots,c_n)$ and $d=(d_i,\ldots,d_n)$.
They are \term{weakly valid} if
\begin{align}\label{eq:offsets}
  d_j-c_i\ge\sij \quad\text{for all $i,\range j1n$},
\end{align}
and \term{valid} if also equality holds in \rf{offsets} for all $(i,j)$ on some transversal $\calT$.
In the latter case, $\calT$ is a HVT and equality holds on all HVTs, \cite[Lemma 3.3]{Pryce2001a}.

Adding the same constant to all $c_i$ and $d_j$ does not change the truth of \rf{offsets}, so without loss, weakly valid offsets can be \term{normalised}, meaning $\min_i c_i=0$.

\item\label{it:Sig5}
If valid offsets exist then necessarily the DAE is SWP, $\calT$ is an HVT, and $\Val(\Sigma)$ is also equal to $\sum_j d_j - \sum_i c_i$.

\item\label{it:Sig6}
If $c,d$ are weakly valid they define an $n\x n$ {\em \sysJ} $\~J(c,d) = (\~J_{ij})$ where (the two formulas are equivalent)
  \begin{align}\label{eq:sysJ}
    \~J_{ij}=
    \begin{cases*}
    \ds\tdbd{\f_i}{x_j^{(\sij)}} & if $d_j{-}c_i{=}\sij$, \\
      0 & otherwise
    \end{cases*}
  =
  \begin{cases*}
    \ds\tdbd{\f_i}{x_j^{(d_j-c_i)}} & if $d_j{-}c_i{\ge}0$, \\
     0 &otherwise.
  \end{cases*}
  \end{align}
  The \Smethod \term{succeeds}, and the DAE is \term{\Same}, if $\~J$ is nonsingular somewhere, see \cite{Pryce2001a} for a precise definition of ``somewhere''.
  Then $\Val(\Sigma)$ is the number of degrees of freedom $\dof$ of the DAE.
  
  If a nonsingular $\~J$ is found then pair $c,d$ must be valid by \lmref{valid_cd}.
  ``Success'' is independent of the valid $c,d$ chosen, since they all give the same $\det\~J$ value, \cite[Theorem 5.1]{nedialkov2005solving}.

\item\label{it:Sig7}
An SWP DAE has \term{canonical offsets}, for which the $c_i$ are elementwise smallest subject to being normalised.
  In case of success we define the \term{structural index} $\nu_s$ of the DAE to be the largest $c_i$ of the canonical offsets.
  
\item\label{it:Sig8} In case of success we have, for any normalised valid $(c,d)$, a {\em standard solution scheme} (SSS) as follows.
If a different $(c,d)$ exists, it just schedules the same operations in a different order, see \cite[Subsection 5.1]{nedialkov2005solving}.

\noindent For $k = k_d := -\max d_j$ to some maximum order $p$
\vspace{-2ex}
\begin{align}\label{eq:SSS}
\begin{aligned}
&  \text{solve the equations} &&\der{\f_i}{k+c_i}=0 &&\text{for those $i$ such that $k+c_i\ge0$} \\[-.75ex]
&  \text{for the unknowns} &&\der{x_j}{k+d_j} &&\text{for those $j$ such that $k+d_j\ge0$}
\end{aligned}\biggr\}\; .
\end{align}
For $k<0$ the SSS applies IVs; for $k\ge0$ it finds the TS to order $p$.

In this paper, the fact that the recurrence defined by \sodes coincides with what the SSS does, is our key to proving that \sodes are a sound approach.
\end{enumerate}
\smallskip

\ssref{extract} uses the following, which we think is not proved elsewhere:
\begin{lemma}\label{lm:valid_cd}
  If offsets are weakly valid and give a nonsingular $\~J$, they are valid.
\end{lemma}
\begin{proof}
  We have $\det\~J\ne0$.
  By the formula for $\det\~J$ as a sum of $\pm$(product of $\~J_{ij}$ values on a transversal $\calT$), there is at least one $\calT$ on which all $\~J_{ij}\ne0$.
  By assumption, the offsets have $d_j-c_i\ge\sij$ for all $(i,j)$.
  Any $(i,j)$ where strict inequality holds gives $\~J_{ij}=0$ in \rf{sysJ}, so $d_j{-}c_i=\sij$ must hold on $\calT$, whence $\calT$ is an HVT.
\end{proof}

\begin{remark}
To find an HVT, one solves a linear assignment problem, e.g.\ by the Jonker--Volgenant algorithm \cite{JonkerVolgenant1987}.
This is a kind of linear programming (LP); many basic \Smethod objects/facts exemplify LP {\em duality}: \eg transversal = primal-feasible; HVT = primal-optimal; weakly valid = dual-feasible; valid = dual-optimal.

\end{remark}

\begin{remark}
  Structural index $\nu_s$ differs from the traditional differentiation index $\nu_d$.
  It is the least number of derivatives of any equation needed, to reduce to an ODE---in some subset of the $n$ variables $x_j$ but whose solution determines all of them.
  E.g., the system $0=\f_1 = x_1-\sin(t), 0=\f_2 = \xp_1-x_2$ needs 1 differentiation of $\f_1$ to find $(x_1,x_2)=(\sin(t),\cos(t))$, so $\nu_s = 1$.
  The ODE is of size zero in this case.
  But $\nu_d=2$, since one must differentiate again to get an ODE for $x_2$.
\end{remark}

\section{Spring-pendulum \cl}\label{ap:springpend}~
The \cl that produced the numerical results for Problem A is 

\medskip

\begin{center}
\small
  \begin{tabular}{rcccrrl|c @{\hskip 0.1cm} c @{\hskip 0.15cm} l}
Line & Kind & Op & Mode & R1 & R2  & Imm & \multicolumn{3}{c}{Expression} \\ \hline 
\texttt{1} & \texttt{ODE} & \texttt{} & \texttt{R} & \texttt{2} & \texttt{} & \texttt{}  & $\dot x_{1}$ &$=$& $x_{2}$ \\ 
\texttt{2} & \texttt{ODE} & \texttt{} & \texttt{R} & \texttt{18} & \texttt{} & \texttt{}  & $\dot x_{2}$ &$=$& $x_{18}$ \\ 
\texttt{3} & \texttt{ODE} & \texttt{} & \texttt{R} & \texttt{4} & \texttt{} & \texttt{}  & $\dot x_{3}$ &$=$& $x_{4}$ \\ 
\texttt{4} & \texttt{ODE} & \texttt{} & \texttt{R} & \texttt{23} & \texttt{} & \texttt{}  & $\dot x_{4}$ &$=$& $x_{23}$ \\ 
\texttt{5} & \texttt{ALG} & \texttt{sub} & \texttt{RI} & \texttt{1} & \texttt{} & \texttt{ 1}  & $x_{5}$ & $=$ & $x_{1} -  1$ \\ 
\texttt{6} & \texttt{ALG} & \texttt{mul} & \texttt{RR} & \texttt{1} & \texttt{4} & \texttt{}  & $x_{6}$ & $=$ & $x_{1} * x_{4}$ \\ 
\texttt{7} & \texttt{ALG} & \texttt{mul} & \texttt{RR} & \texttt{6} & \texttt{4} & \texttt{}  & $x_{7}$ & $=$ & $x_{6} * x_{4}$ \\ 
\texttt{8} & \texttt{SUB} & \texttt{cs} & \texttt{RR} & \texttt{10} & \texttt{3} & \texttt{}  & \multirow{2}{*}{$\begin{bmatrix} x_{8} \\ x_{9} \end{bmatrix}$} & \multirow{2}{*}{$=$} & \multirow{2}{*}{$\begin{bmatrix} x_{10} \\ x_{8} \end{bmatrix} \odot_{\text{cs}} x_{3}$} \\ 
\texttt{9} & \texttt{SUB} & \texttt{} & \texttt{RR} & \texttt{8} & \texttt{3} & \texttt{}  &  \\ 
\texttt{10} & \texttt{ALG} & \texttt{sub} & \texttt{IR} & \texttt{} & \texttt{9} & \texttt{ 0}  & $x_{10}$ & $=$ & $ 0 - x_{9}$ \\ 
\texttt{11} & \texttt{ALG} & \texttt{mul} & \texttt{IR} & \texttt{} & \texttt{8} & \texttt{ 9.81}  & $x_{11}$ & $=$ & $ 9.81 * x_{8}$ \\ 
\texttt{12} & \texttt{ALG} & \texttt{add} & \texttt{RR} & \texttt{7} & \texttt{11} & \texttt{}  & $x_{12}$ & $=$ & $x_{7} + x_{11}$ \\ 
\texttt{13} & \texttt{ALG} & \texttt{add} & \texttt{RI} & \texttt{5} & \texttt{} & \texttt{ 1}  & $x_{13}$ & $=$ & $x_{5} +  1$ \\ 
\texttt{14} & \texttt{ALG} & \texttt{sub} & \texttt{IR} & \texttt{} & \texttt{5} & \texttt{ 0}  & $x_{14}$ & $=$ & $ 0 - x_{5}$ \\ 
\texttt{15} & \texttt{SUB} & \texttt{exp} & \texttt{RR} & \texttt{15} & \texttt{14} & \texttt{}  & $ x_{15}$ &$=$ & $x_{15} \odot_\text{exp} x_{14}$ \\ 
\texttt{16} & \texttt{ALG} & \texttt{sub} & \texttt{RR} & \texttt{13} & \texttt{15} & \texttt{}  & $x_{16}$ & $=$ & $x_{13} - x_{15}$ \\ 
\texttt{17} & \texttt{ALG} & \texttt{mul} & \texttt{IR} & \texttt{} & \texttt{16} & \texttt{ 40}  & $x_{17}$ & $=$ & $ 40 * x_{16}$ \\ 
\texttt{18} & \texttt{ALG} & \texttt{sub} & \texttt{RR} & \texttt{12} & \texttt{17} & \texttt{}  & $x_{18}$ & $=$ & $x_{12} - x_{17}$ \\ 
\texttt{19} & \texttt{ALG} & \texttt{mul} & \texttt{IR} & \texttt{} & \texttt{2} & \texttt{-2}  & $x_{19}$ & $=$ & $-2 * x_{2}$ \\ 
\texttt{20} & \texttt{ALG} & \texttt{mul} & \texttt{RR} & \texttt{19} & \texttt{4} & \texttt{}  & $x_{20}$ & $=$ & $x_{19} * x_{4}$ \\ 
\texttt{21} & \texttt{ALG} & \texttt{mul} & \texttt{IR} & \texttt{} & \texttt{9} & \texttt{ 9.81}  & $x_{21}$ & $=$ & $ 9.81 * x_{9}$ \\ 
\texttt{22} & \texttt{ALG} & \texttt{sub} & \texttt{RR} & \texttt{20} & \texttt{21} & \texttt{}  & $x_{22}$ & $=$ & $x_{20} - x_{21}$ \\ 
\texttt{23} & \texttt{ALG} & \texttt{div} & \texttt{RR} & \texttt{22} & \texttt{1} & \texttt{}  & $x_{23}$ & $=$ & $x_{22} / x_{1}$ \\ 
\end{tabular}
\end{center}
\clearpage


%

\end{document}